\documentclass[11pt,a4paper]{article}
\usepackage{color,soul}
\usepackage{amsmath,amssymb,amsthm,latexsym}
\usepackage{bbm}
\usepackage{amsfonts}
\usepackage[english]{babel}
\usepackage{graphicx}
\usepackage{enumitem}
\usepackage{comment}
\usepackage{authblk}
\usepackage[autostyle]{csquotes}
\usepackage{circuitikz}
\usepackage{tikz,graphicx,subfigure}
\usetikzlibrary{hobby}
\usepackage{hyperref}
\usepackage[a4paper, total={6in, 9in}]{geometry}
\usepackage[numbers,sort]{natbib}
\usepackage{todonotes}
\hypersetup{
    colorlinks=true,
    linkcolor=black,
    filecolor=magenta,
    citecolor=black,
    urlcolor=cyan,
    pdftitle={BFKL25},
    pdfpagemode=FullScreen,
    }
\usetikzlibrary{decorations.markings}
\usetikzlibrary{arrows}

\newcommand{\SL}[1]{{\color{red}{#1}}} 
 
\newtheorem{theorem}{Theorem}

\newtheorem{lemma}[theorem]{Lemma}
\newtheorem{proposition}[theorem]{Proposition}
\newtheorem{corollary}[theorem]{Corollary}

\theoremstyle{definition}
\newtheorem{definition}[theorem]{Definition}
\newtheorem{rhproblem}[theorem]{RH problem}
\newtheorem{remark}[theorem]{Remark}

\allowdisplaybreaks[4]

\let\Re\undefined
\let\Im\undefined
\DeclareMathOperator{\Re}{Re}
\DeclareMathOperator{\Im}{Im}

\numberwithin{equation}{section}
\numberwithin{theorem}{section}

\def\XXint#1#2#3{{\setbox0=\hbox{$#1{#2#3}{\int}$}
		\vcenter{\hbox{$#2#3$}}\kern-.5\wd0}}

\date{}                     

\begin{document}

\title{Orthogonal polynomials in the spherical ensemble with two insertions}


\author[1]{Sung-Soo Byun}
\affil[1]{\small Department of Mathematical Sciences and Research Institute of Mathematics, Seoul National University, Seoul 151-747, Republic of Korea}

\author[2]{Peter J. Forrester}
\affil[2]{School of Mathematical and Statistics, The University of Melbourne, Victoria 3010, Australia}
\author[3]{Arno B.J. Kuijlaars}

\author[2,3]{Sampad Lahiry}
\affil[3]{
Department of Mathematics, Katholieke Universiteit Leuven, Celestijnenlaan
200 B bus 2400, 3001 Leuven,  Belgium}

\affil[  ]{sungsoobyun@snu.ac.kr; \,pjforr@unimelb.edu.au; \ arno.kuijlaars@kuleuven.be; \, sampad.lahiry@kuleuven.be}

\maketitle

\begin{abstract}
We consider asymptotics of planar orthogonal polynomials $P_{n,N}$ (where $\mathrm{deg}P_{n,N}=n$) with respect to the weight $$\frac{|z-w|^{2NQ_1}}{(1+|z|^2)^{N(1+Q_0+Q_1)+1}}, \quad(Q_0,Q_1 > 0)$$ in the whole complex plane. With $n, N\rightarrow\infty$
and $N-n$ fixed,
we obtain the strong asymptotics of the polynomials, asymptotics for the weighted $L^2$ norm and the limiting zero counting measure.
These results apply to the pre-critical phase of the underlying two-dimensional Coulomb gas system, when the support of the equilibrium measure  is simply connected. Our method relies on specifying the mother body of the two-dimensional potential problem. It relies too
on the fact that the planar orthogonality can be rewritten as a non-Hermitian contour orthogonality. This allows us to perform the Deift-Zhou steepest descent analysis of the associated $2\times 2$ Riemann-Hilbert problem.
\end{abstract}

\tableofcontents

\section{Introduction and statement of results}

\subsection{Underlying Coulomb gas, random matrix ensemble and relation to planar orthogonal polynomials}

Let $\{ \mathbf{r}_j \}_{j=1}^N$ be $N$ points on the surface of the sphere of radius $\frac{1}{2}$ embedded in $\mathbb R^3$ and centred at the origin. Let $|| \mathbf x ||$ denote the usual Euclidean length in $\mathbb R^3$. Suppose these $N$ unit points are
in fact charges obeying Poisson's equation as defined on the surface of the sphere. This means that a point at $\mathbf x$ and a point at
$\mathbf x'$ interact via the pair potential $- \log || \mathbf x - \mathbf x' ||$; see e.g.~\cite[Eq.~(15.105)]{Fo10}. In addition, let
there be two fixed charges of (macroscopic) strengths $NQ_0$ and $NQ_1$ at positions $\mathbf R_0$ and $\mathbf R_1$, with $Q_0,Q_1 > 0$. Up to a multiplicative
constant, the corresponding Boltzmann factor has the functional form
\begin{equation}\label{F1}
||  \mathbf R_0 - \mathbf R_1||^{\beta N^2 Q_0 Q_1}   \prod_{l=1}^N || \mathbf  R_0  - \mathbf r_l ||^{\beta N Q_0} || \mathbf  R_1  - \mathbf r_l ||^{\beta N Q_1} 
\prod_{1 \le j < k \le N} || \mathbf r_k - \mathbf r_j ||^\beta,
\end{equation}
where $\beta >0$ is the inverse temperature.
By rotational invariance of the sphere, it is always possible to choose $\mathbf R_0$ to be at the south pole $(0,0,-\frac12)$, which we henceforth assume.

From the viewpoint of charge neutrality in the Coulomb gas picture, it has been noted in  \cite{BFL25} that associated with fixed charges at positions
$\mathbf R_0$ and $\mathbf R_1$ are spherical caps of area $ \pi Q_0/(1+Q_0 + Q_1 )$ and $ \pi Q_1/(1+Q_0 + Q_1 )$, respectively. 
As already evident from potential theoretic viewpoints  \cite{CC03,BDSW18,CK19,LD21,CK22} the Coulomb gas exists in two phases depending on there being no overlap of the spherical caps (referred to as the \textit{post-critical phase}), or the spherical caps overlapping (referred to as the \textit{pre-critical phase}).
With $(\theta, \phi)$ the polar angle, azimuthal angle pair corresponding to the point $\mathbf R_1$, we have from \cite[Eq.~(1.7)]{BFL25} that the condition for no
overlap, and thus the post-critical phase is
\begin{equation}\label{F1a}
 \cos^2 (\theta/2)  >   {1 \over (\gamma_1 + \gamma_2 + 2)^2} \Big ( \sqrt{(\gamma_1 + 1) (\gamma_1 + \gamma_2 + 1)} + \sqrt{\gamma_2 + 1} \Big )^2 
\Big |_{\substack{\gamma_1 = -1+ Q_0/Q_1 \\  \gamma_2 =1/Q_1}}.
\end{equation}
Conversely, if this inequality is in the other direction, the Coulomb gas system is in the pre-critical phase.

Suppose now that the points on the sphere are stereographically projected to the complex plane, viewed to be tangent to the north pole; this construction is illustrated e.g.~in \cite[Figure 15.2]{Fo10}. With $\mathbf r$, $\mathbf r'$ points on the sphere and $z$, $z'$ the corresponding points on the plane, we then have that
\cite[equivalent to Eq. (15.126)]{Fo10}
\begin{equation}\label{F2}
||  \mathbf r - \mathbf r' || = {|z - z' | \over (1 + |z|^2)^{1/2}  (1 + |z'|^2)^{1/2}}.
\end{equation} 
We have too the relation between the differential on the surface of the sphere $d S$ say, and the flat measure on the complex plane viewed as $\mathbb R^2$
\cite[Eq.~(15.127)]{Fo10}
\begin{equation}\label{F3}
d S = {1 \over (1 + |z|^2)^2} \, dx \, dy.
\end{equation} 
 These two facts together give that the stereographically projected form of 
 (\ref{F1}) is equal to
 \begin{equation}\label{F4}
{1 \over (1 + |w|^2)^{\beta Q_0 Q_1 N^2} }   \prod_{l=1}^N { | w - z_l|^{\beta Q_1 N}  \over (1 + |z_l|^2)^{\beta (1+Q_0 + Q_1  ) N/2+2-\beta/2 } }
\prod_{1 \le j < k \le N} | z_k - z_j |^\beta,
\end{equation} 
where $w$ is the stereographic projection onto the complex plane of $\mathbf R_1$. It is convenient to rotate the sphere about the veritical axis so that
$w > 0$.
The equation \eqref{F1a} can be rewritten in terms of $w$. Thus the
 requirement that the spherical caps associated with two charges overlap restricts $Q_0,Q_1,w$ to be such that \cite[Eq.~(2.74)]{BFL25}
 \begin{equation}  \label{def of w critical point}
 w  > w_{\rm cri} := \Big( 2Q_0Q_1+Q_0+Q_1 + 2\sqrt{ Q_0 Q_1(1+Q_0)(1+Q_1) } \Big)^{-1/2}    .
 \end{equation}

Now set
 \begin{equation}\label{F5}
 \beta = 2, \quad Q_1 N = r, \quad Q_0 N = K
 \end{equation} 
 in (\ref{F4}). The  factor in (\ref{F4}) independent of $w$ is then
 \begin{equation}\label{F6}
 \prod_{l=1}^N {1 \over (1 + |z_l|^2)^{ K + r +  N + 1 } }
\prod_{1 \le j < k \le N} | z_k - z_j |^2.
\end{equation}
As revised in \cite[\S 3.1]{BFL25} this functional form, up to proportionality, is the eigenvalue probability density function for the particular random matrix ensemble SrUE${}_{(N,K+r)}$. Here the (non-Hermitian) ensemble SrUE${}_{(N,s)}$ is formed from random matrices $(G^\dagger G)^{-1/2} X$ where $G$ is an $(N+s) \times N$ complex standard Gaussian matrix and $X$ is an $N \times N$  complex standard Gaussian matrix; the case $s=0$ corresponds to the complex spherical ensemble \cite[\S 2.5]{BF25}. This allows (\ref{F4}) with the substitution (\ref{F5}) to be identified as 
proportional to 
 \begin{equation}\label{F7}
{1 \over (1 + |w|^2)^{2 r K} } \Big \langle   \prod_{l=1}^N  | w - z_l|^{2r} \Big \rangle_{{\rm SrUE}_{(N,K+r)}}.
\end{equation}
In words the second factor is the ensemble average with respect to ${\rm SrUE}_{(N,K+r)}$ of the $2r$-th power of the absolute value of the characteristic polynomial $ \prod_{l=1}^N   | w - z_l|$. A result from \cite{Fo25} gave that this ensemble average
is, assuming $K \ge r$, equal to the different ensemble average
 \begin{equation}\label{F8}
 \Big \langle   \prod_{l=1}^r   ( |w|^2  + t_l)^{N} \Big \rangle_{{\rm JUE}_{r,(0,K-r)}}.
 \end{equation}
 Here ${\rm JUE}_{r,(0,K-r)}$ denotes the (Hermitian) Jacobi unitary ensemble of $r$ eigenvalues $\{ t_l \}$ supported on the interval $(0,1)$
 and distributed according to the probability density function proportional to
 $$
 \prod_{l=1}^r (1 - t_l)^{K-r} \prod_{1 \le j < k \le n} | t_k - t_j|^2;
 $$
 for more on this see \cite[Ch.~3]{Fo10}. Since the roles of $N$ and $r$ in the average in (\ref{F7}) and that of (\ref{F8}) are reversed their equality is referred to
 as a duality identity. In \cite{BFL25} this was used as one tool to study the large $N$ form of the configuration integral associated with the $\beta = 2$
 case of the Boltzmann  factor (\ref{F1}) or equivalently (\ref{F4}), in both the pre- and post-critical phases.

A salient feature of (\ref{F1})  and  (\ref{F4}) with $\beta = 2$ is that they  both specify determinantal point processes. Consider for definiteness  (\ref{F4}). Excluding the first factor, this can be written
  \begin{equation}\label{F9}
  \prod_{l=1}^N e^{-N V(z_l)} \prod_{1 \le j < k \le N }|z_k - z_j|^2,
  \end{equation}
  where
 \begin{equation}\label{F10}
    V(z)=\Big(1+\frac{1}{N}+Q_0+Q_1\Big)\log (1+|z|^2)-2Q_1 \log |z-w|.
\end{equation}  
Introduce monic orthogonal polynomials $\{ P_{k,N}(z) \}_{k=0}^{N-1}$ of degree $k$ such that
 \begin{equation}\label{F11}
\int_{\mathbb C}P_{l,N}(z)\overline{P_{k,N}(z)}e^{-N V(z)} \, dA(z)=h_{k,N}\, \delta_{k,l}, \quad (l,k=0,1\dots,N-1),
\end{equation}
where $dA(z) = dx \,  dy$. One notes that $V(z)$ and thus the polynomials $P_{k,N}(z) $ depend on $w$, although this 
dependance is suppressed in the notation. Standard theory (see e.g.~\cite[\S 15.3]{Fo10}) gives that
 \begin{equation}\label{F12}
 \int_{\mathbb C}  dA(z_1) \cdots  \int_{\mathbb C}  dA(z_N) \,   \prod_{l=1}^N e^{-N V(z_l)} \prod_{1 \le j < k \le N} |z_k - z_j|^2
 = N! \prod_{l=0}^{N-1} h_{l,N},
 \end{equation}
 and
  \begin{equation}\label{F13}
  \rho_{(k)}(z_1,\dots,z_k)
  =\det[K_{N}(z_i,z_j)]_{i,j=1}^{k},
   \end{equation} 
with 
 \begin{equation}\label{F14}
K_{N}(w,z)=e^{-\frac{N}{2}\left(V(w)+V(z)\right)}\sum_{l=0}^{N-1}\frac{P_{l,N}(w)\overline{P_{l,N}(z)}}{h_{l,N}}.
 \end{equation}
 In (\ref{F13}), $ \rho_{(k)}$ denotes the $k$-point correlation function, defined as a suitable normalisation times (\ref{F9}) integrated over $z_{k+1},\dots,z_N$; see e.g.~\cite[\S 5.1.1]{Fo10}. Of particular interest is the large $N$ form of the partition function (\ref{F12}), which will be dependent on the phase (post or pre-critical), as well as the correlation function (\ref{F14}), which in fact can be proved to exhibit universal properties \cite{AHM11,AHM15,HW21} in the microscopic scaling regime.
The study of \cite{HW21} (and its subsequent work \cite{Hed24}) is particularly relevant to the present work, as it derives the asymptotic form of the orthonormal polynomials $h_{n,N}^{-1/2} P_{n,N}(z)$ outside and on the boundary of the droplet in the large $N, n$ limit, with $n/N \in (0,1]$ fixed. This result holds for a general class of potentials $V(z)$ under certain conditions, such as the connectivity of the droplet. Here, the droplet refers to the support of the equilibrium measure associated with $V$; see e.g. \cite{ST97}.    
 
 In the present work we take up the problem of obtaining asymptotics of $ h_{n,N}$ and $P_{n,N}(z)$ separately
 in the case that $V(z)$ is given by (\ref{F10}), using
 ideas introduced in the work \cite{BBLM15}, in the case that  
 \begin{equation}\label{F10a}
    V^{\rm g}(z)= |z|^2 -2Q \log |z-w|.
\end{equation}  
The term $-2Q \log |z-w|$ is the potential due to an external charge, the effect of which in two-dimensional Coulomb gas systems has been studied in the recent works
\cite{B24,BCMS,BKS23,AKS23,KLY23,BY25,BY23,DS22,WW19,BSY24,BKSY25}.
A preliminary to obtaining the asymptotics requires developing a mother body theory relating to the equilibrium measure associated with (\ref{F9}), which is of independent interest. 

\subsection{Statement of results}
For a compactly supported finite Borel measure $\mu$ we denote the logarithmic potential and corresponding logarithmic energy as
\begin{equation} \label{def of U mu and I mu}
 U^{\mu}(z)=\int \log \frac{1}{|z-s|}\, d\mu(s),\qquad  I(\mu)=\int U^{\mu}(z)\, d\mu(z).
\end{equation}
The equilibrium measure associated with the large-$N$ limit of the potential \eqref{F10} 
 corresponds to the unique probability measure $\nu_0$ which minimizes the energy functional
\begin{equation}\label{enf}
I(\nu)-2Q_1\int \log|w-u| \, d\nu(u)+(1+Q_0+Q_1)\int \log (1+|u|^2) \, d\nu(u) 
\end{equation}
over all probability measures supported on the complex plane.
The external field is strongly admissible (i.e.~has has sufficient growth at infinity \cite{ST97}), hence the equilibrium measure is compactly supported.  We denote 
$\mathrm{supp}(\nu_0)=\Omega$.
For the equilibrium measure, it is well known \cite{ST97} that there exists a real constant $\ell_{\mathrm{2D}}$ (the precise form is given in \cite[Prop.~2.9]{BFL25}) such that the variational conditions 
\begin{equation}\label{frost1}
    \mathcal{U}_{2D}(z)=2U^{\nu_0}(z)-2Q_1 \log |z-w|+\left(1+Q_0+Q_1\right)\log (1+|z|^2)+\ell_{\mathrm{2D}}\begin{cases}=0\quad z\in \Omega,\\
    \geq 0 \quad z\in \mathbb C\setminus \Omega
    \end{cases}
\end{equation}
hold.

The particular (connected) domain $\Omega$ solving (\ref{frost1}) in the pre-critical phase has been determined in \cite{BFL25}.
There (in Prop.~2.7) the parameters $\rho, a, b$ of the conformal map
\begin{equation}\label{confm}
        f(u):= \frac{\rho}{u}\frac{1-bu}{1-au},
    \end{equation}
  from the interior of the unit disc $\mathbb{D}$ to the exterior of the droplet   $\mathbb C \setminus \Omega$ have been specified. The details are not required
  in the present work, except for the  particular features 
  \begin{equation}\label{parameters1}
    \rho > 0  , \qquad 0<a<1, \qquad a<b.
  \end{equation} We also point out that $b<1$ if and only if $0$ lies in the interior of the droplet. Another key feature of these parameters that we will use is   \cite[Eq.~(2.60)]{BFL25}
  \begin{equation} \label{a b ratio Q0Q1}
      \frac{b\rho^2}{a}=\frac{1+Q_1}{Q_0}.
  \end{equation}
  One comments too that 
  $\nu_0$ is absolutely continuous with respect to the two-dimensional Lebesgue measure $dA(u)$ with density given by
    \begin{equation}
        \frac{d\nu_0(u)}{dA(u)}=\frac{(1+Q_0+Q_1)}{\pi}\frac{1}{(1+|u|^2)^2}\, \mathbbm{1}_{\Omega}(u),
    \end{equation}
where $\mathbbm{1}_{\Omega}$ is the characteristic function on $\Omega$.

\begin{figure}
    \centering\includegraphics[width=1\linewidth]{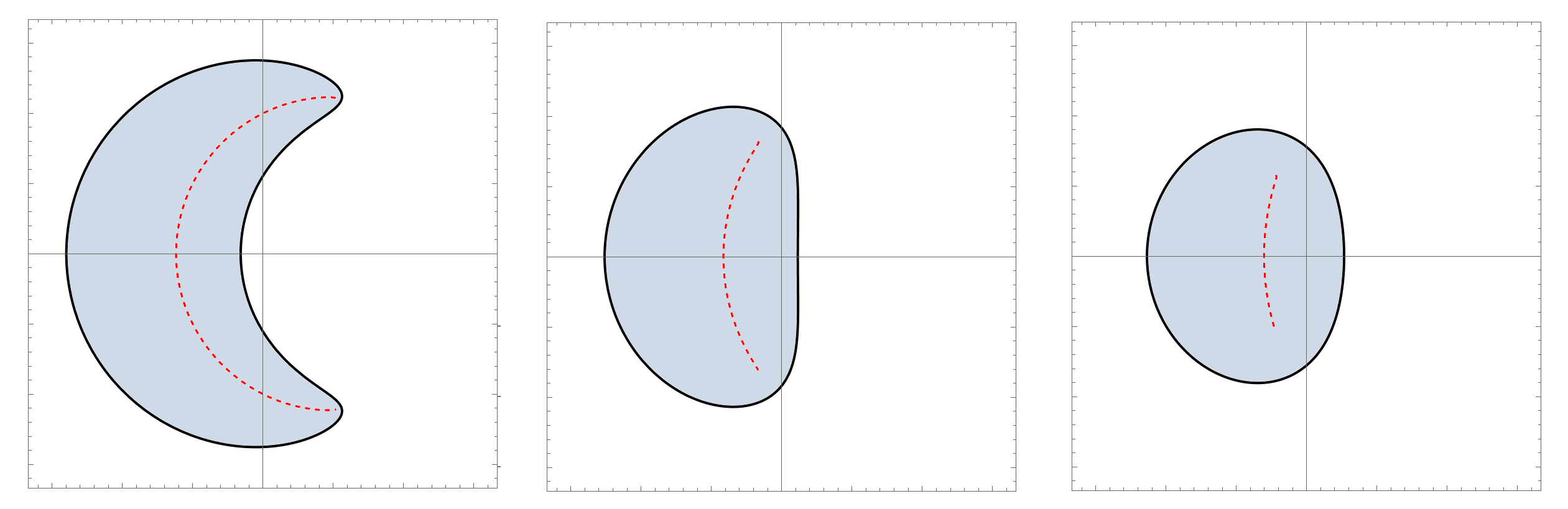}
    \caption{Evolution of the droplet (shaded blue) and the mother body (dashed red) as $w$ increases. The figures correspond to $Q_0=Q_1=1$ and $w=0.5 , 1, 2$ (from left to right). }
    \label{Dropmoth}
\end{figure}

Our first result is concerning the mother body (see e.g.~the introduction of \cite{BS16}, \cite{GTV14}  for contextual and historical remarks). We construct a one-dimensional measure supported on a curve $\Gamma_0$ whose Cauchy (also known as Stieltjes) transform agrees with the Cauchy transform of the two-dimensional equilibrium 
measure on the exterior of $\Omega$.

\begin{theorem}\label{thmmoth}
Suppose $w> w_{\rm cri}$ (see \eqref{def of w critical point}). There exists a Borel probability measure $\mu_0$, supported on a curve $\Gamma_0$, which lies in the interior of $\Omega$ and intersects the real line between $(-1/w,0)$  with the following properties: 
    \begin{itemize}
        \item[(1)]  There is a rational function $R_0$ (explicitly given in \eqref{defR0}) such that
    \begin{equation}\label{mothdens}
        d\mu_0(s)=\frac{1}{2\pi i}\sqrt{R_0(s)}\, ds,
    \end{equation}
    where $ds$ represents the complex line element of  $\Gamma_0$. The branch cut of the square root of $R_{ 0 }(s)$ is taken in $\Gamma_0$ and the branch is taken such that 
    \begin{equation} \label{R0 sqrt asymp z infty}
       \sqrt{R_0(z)}=\frac{Q_0}{z}+\mathcal{O}(z^{-2}), \quad \text{as } z \to \infty. 
    \end{equation}
    \item [(2)] There is a smooth simple closed curve $\Gamma$, which has $[0,w]$ in its interior and $(-\infty,-1/w]$ in its exterior, with $\Gamma_0\subset \Gamma$ and a real constant $\ell_0$ such that
    \begin{equation}\label{eqcon}
    2U^{\mu_0}(z)+\Re \mathcal{V}(z)+\ell_0\begin{cases}=0 \quad z\in \Gamma_0,\\
    >0  \quad z\in \Gamma\setminus  \Gamma_0,
    \end{cases}
\end{equation}
where 
\begin{equation}\label{defV}
    \mathcal{V}(z)=(1+Q_1)\log z+(1+Q_0) \log \Big(z+\frac{1}{w}\Big)-Q_1\log (z-w) .
\end{equation}
Here, the branch of the logarithm is chosen on $(-\infty,w]$.
\item [(3)] There is the equality of Cauchy transforms
    \begin{equation}\label{moth3}
        \int \frac{d\nu_0(s)}{z-s}=  \int \frac{d\mu_0(s)}{z-s}, \quad z\in \mathbb C\setminus \Omega. \end{equation}
    \end{itemize}

\end{theorem}

To obtain the large $N$ asymptotics of the orthogonal polynomials $P_{n,N}$, we rely on the fact that the planar orthogonality can be reformulated as non-Hermitian contour orthogonality,
a fact which is of independent interest.

\begin{proposition}
    \label{contorth}
Let $P_{n,N}(z)$ satisfy 
  \begin{equation}\label{eq9}
    \int_{\mathbb C}P_{n,N}(z)\overline{(z-w)}^{k} 
    \frac{|z-w|^{2 N Q_1} }{(1+|z|^2)^{N(1+Q_0+Q_1)+1}}\, dA(z)=0, \quad k=0,\dots n-1.
\end{equation}
Then for a simple closed  contour $\Gamma$ with positive orientation which has $[0,w]$ in its interior and $(-\infty,-1/w]$ in the exterior we have
\begin{equation}\label{nonh}
    \oint_{\Gamma}P_{n,N}(z)z^{j}\frac{(z-w)^{NQ_1}}{(1+zw)^{N+NQ_0}}\frac{dz}{z^{n+NQ_1}}=0,\quad j=1,\dots n-1.
\end{equation}
\end{proposition}
\begin{remark}
  One observes that as $z\rightarrow\infty$, the integrand in \eqref{eq9} is $\mathcal{O}(z^{2(n-N-NQ_0)-3})$.  
  Then for the integral to converge we want this to be $\mathcal{O}(z^{-3})$. This is the case if we have $n\leq N+ NQ_0$. Hence the orthogonal polynomials of all degrees less than $n$ exists for all $n$ such that 
    \begin{equation}\label{r0cri}
    n\leq N+ NQ_0.
    \end{equation}
  In what follows we will fix $n-N=r_0$ and will let $n,N\rightarrow\infty$.  Since $Q_0>0$, due to \eqref{r0cri} we can then fix $r_0$ to be any real number.
\end{remark} 
For the orthogonal polynomials associated with the external potential \eqref{F10a}, the equivalent contour orthogonality was established in \cite[\S 3]{BBLM15} and further extended in \cite{LY19,BKP23} to the case of multiple point charges. Such contour orthogonality serves as a cornerstone for several works \cite{BGM17,BEG18,BKP23, BSY24,BY23, KLY23,LY17,LY23,WW19} in various asymptotic analyses of orthogonal polynomials, as well as in the study of statistical properties of the associated Coulomb gases.
Since \eqref{F10a} can be obtained as a particular limit of \eqref{F10} (see \cite{BFL25} for a related discussion), Proposition~\ref{contorth} can be regarded as an extension of the result in \cite[\S 3]{BBLM15}.

\begin{remark}
    From \eqref{nonh}, the weight function for the non-Hermitian contour orthogonality is given as 
    \begin{equation}\label{wC}
    w_{n,N}(z)=\frac{(z-w)^{NQ_1}}{(z+1/w)^{N+NQ_0}}\frac{1}{z^{n+NQ_1}}.
   \end{equation} 
Then \eqref{wC} is analytic in $\mathbb C \setminus ((-\infty,-1/w]\cup [0,w])$ also if $N,NQ_0,NQ_1$  are not integers. This matches with $e^{-N \mathcal{V}(z)}$ (recall \eqref{defV} for definition of $\mathcal{V}$) up to a factor $z^{n-N}$.  
This will play a crucial role in the steepest descent analysis that follows.
\end{remark}

By \cite{FIK92}, Proposition~\ref{contorth} further allows us to write $P_{n,N}$ as a solution to a $2\times 2$ Riemann-Hilbert problem which is described at the beginning of Section 5, referred to as RH problem 5.1.
To obtain the strong asymptotics of $P_{n,N}$
we proceed in Section 5 to perform the Deift-Zhou steepest descent analysis \cite{DZ93} of the
Riemann-Hilbert problem, which at a technical level has steps in common with the works \cite{BBLM15, BGM17,BK12, DKMVZ99, LY19, LY23, KKL24, LY17, BEG18, BY23, BSY24}.
To state our main result in this regard, we recall the measure $\mu_0$ in Theorem \ref{thmmoth} and its associated $g$-function, 
\begin{equation}
    g(z)=\int \log(z-s) \, d\mu_0(s).
\end{equation}
Further, we denote the conformal map $F_1$ which maps the exterior of the droplet $\Omega$ to the unit disc $\mathbb D(0,1)$. It is the inverse of the analytic function $f(u)$ given in \eqref{confm}, which behaves as $F_1(z)=\rho z^{-1}+\mathcal{O}(z^{-2})$ as $z\rightarrow\infty$.

\begin{theorem}\label{thm1}
    Let $w$ satisfy \eqref{def of w critical point} and let $n,N\rightarrow \infty$ be such that $n-N=r_0$ is any fixed real number. Then we have 
\begin{equation}\label{strongasy}
    \begin{aligned}
     P_{n,N}(z)=i \Big(\frac{\rho}{F_1(z)}\Big)^{r_0}\frac{\sqrt{\rho F_1'(z)}}{F_1(z)}e^{Ng(z)}(1+\mathcal{O}(N^{-1})) \quad \text{as}\hspace{0.1cm} N \rightarrow\infty
    \end{aligned}
\end{equation}
uniformly for $z$ in compact subsets of $\overline{\mathbb C}\setminus\Gamma_0 $. Here the branch of the square root in $\sqrt{F_1'(z)}$ is taken such that $i  \sqrt{\rho F_1'(z)}/F_1(z)$ tends to $1$ as $z\rightarrow\infty$.
\end{theorem}

\begin{remark} $ $

\begin{enumerate}
\item In \eqref{strongasy}, the pre-factor in front of $e^{N g(z)}$ 
appears as the $(1,1)$ entry of the  RH problem \ref{RHM}.

\item The map $F_{1}(z)$ originally defined in $\mathbb C \setminus \Omega$ has an analytic  continuation to the domain $\mathbb C \setminus \Gamma_0$. Moreover, we will also see that $F_{1}(z)\neq 0$ and analytic in $\mathbb C \setminus \Gamma_0$. This analytic continuation is used in \eqref{strongasy}.

\item The map 
    \begin{equation}
        F_{0}(z)=\frac{1}{F_1(z)}:\mathbb C\setminus \Omega\rightarrow \mathbb C\setminus \mathbb D
    \end{equation}
    is the conformal map from the exterior of the droplet to the exterior of the unit disc. Then \eqref{strongasy} can be rewritten as \begin{equation}\label{exp2}
    P_{n,N}(z)=\left(\rho F_{0}(z)\right)^{r_0}e^{Ng(z)} \sqrt{\rho F_{0}'(z)}(1+\mathcal{O}(N^{-1})) .
\end{equation}
The expansion \eqref{exp2} is similar to the strong asymptotics of the orthogonal polynomials obtained in \cite{BBLM15}, and consistent with the known expansion of \cite{HW21} in the exterior and the boundary of the droplet. 
\end{enumerate}

\end{remark}

As a consequence of Theorem \ref{thm1}, we obtain the limiting zero counting
measure of the polynomials $P_{n,N}$.

\begin{theorem}\label{thm2}
    Under the assumptions of Theorem \ref{thm1} all zeros of $P_{n,N}$ tend to $\Gamma_0$ as $n\rightarrow\infty$. In addition, $\mu_0$ is the weak limit of the normalized zero counting measure of $P_{n,N}$.
\end{theorem}

We also obtain the large $N$ expansion of the normalization constants $h_{n,N}$ in \eqref{F11}. Recall $\rho$ in \eqref{confm} and $\ell_{\mathrm{2D}}$ in \eqref{frost1}.

\begin{theorem}\label{norma}
    Under the same assumption of Theorem \ref{thm1}, we have 
  \begin{equation}\label{norma2}
   h_{n,N}=\pi\sqrt{\frac{2\pi}{N(1+Q_0+Q_1)}} 
  \rho^{2r_0+1} e^{N\ell_{\mathrm{2D}}}\left(1+\mathcal{O}(N^{-1})\right).
\end{equation}  
\end{theorem}

\begin{remark} ${}$ 
\begin{enumerate}
  \item We find the structure exhibited in Theorem \ref{norma} remarkable. The RH problem \ref{rhproblemY} gives us the asymptotics of the normalization constants of $P_{n,N}$ with respect to the weight $w_{n,N}$, and this involves the Robin constant $\ell_0$ in \eqref{eqcon}. However, we can express $\ell_0$ in terms of $\ell_{\mathrm{2D}}$ in a curious way, as is done in Lemma \ref{lemmrela} (a similar miraculous relation also appears in \cite[Lemma~7.2]{BBLM15}). As we shall see, this drastically simplifies $h_{n,N}$ and gives us \eqref{norma2}. 
   \item If (\ref{norma2}) should be uniformly valid in $r_0$ for large $N$, substitution in
    (\ref{F12})  predicts the leading large $N$ asymptotic form
     \begin{equation}\label{n3}
     \log Q_N \sim N^2 ( \ell_{\rm 2 D} - \log \rho ).
   \end{equation}
   Here $Q_N$ is the configuration integral associated with (\ref{F9}), as given in the LHS of (\ref{F12}).
   Using the identification of $\rho$ in (\ref{confm}) as the logarithmic capacity of the
   droplet (see e.g.~\cite[text below (1.24)]{BBLM15}),
   $$
   \log \rho = \int_{\Omega} d \mu(z) \int_{\Omega} d \mu(w) \, \log | w - z|,
   $$
   we see that (\ref{n3}) is consistent with the electrostatic argument of \cite[\S 2.4]{BFL25}.
   \end{enumerate}
\end{remark}

\subsection{Overview of the rest of the paper}
In Section~\ref{schwarzsection}, we construct a degree~$3$ meromorphic function on a genus-zero Riemann surface~$\mathcal{R}$, which we refer to as the \emph{spherical Schwarz function}. In Section~\ref{mbsection}, this function is used to construct the mother body, leading to the proof of Theorem~\ref{thmmoth}. Next, in Section~\ref{eqsection}, we establish the equivalence between planar orthogonality and non-Hermitian contour orthogonality, thereby proving Proposition~\ref{contorth}. Building on this equivalence, Section~\ref{RHsection} presents a Deift--Zhou steepest descent analysis, through which Theorem~\ref{thm1} is obtained. Finally, in Section~\ref{orthsection}, we conclude with the proof of Theorem~\ref{thm2}.

\section{Construction of a meromorphic function on a Riemann surface}\label{schwarzsection}

The map $f$ in \eqref{confm} is a rational map of degree $2$ with poles at $1/a$ and $0$ and zeros at $1/b$ and $\infty$. In addition, the equation $f'(u)=0$ admits two complex conjugate solutions, which are denoted by $u_1,u_2$.

In terms of the parameters $a,b$ of $f$, they are given as 
\begin{equation}
    u_1=\frac{a}{b}-\frac{i}{b}\sqrt{\frac{b}{a}-1},\qquad   u_2=\frac{a}{b}+\frac{i}{b}\sqrt{\frac{b}{a}-1}.
\end{equation}
Due to \eqref{parameters1} we have thus verified that $u_1$ and $u_2$ are complex valued.
The corresponding critical values \( z_j=f(u_j) \) are also complex conjugates.  Explicitly they are given as
\begin{equation}
    z_1=\rho\Big( 2 a -b +2i  \sqrt{a (b-a)}\Big),\qquad z_2=\rho\Big( 2 a -b -2i  \sqrt{a (b-a)}\Big).
\end{equation}
 We also note at this stage that the equation $f(z)f(1/z)+1=0$ permits two solutions, $v_0$, $1/v_0$ both of which are real with $0<v_0<1$ and
\begin{equation}\label{speq1}
f(v_0)=w, \qquad f(1/v_0) = - 1/w;
\end{equation}
see \cite[\S 2.3]{BFL25}.

\begin{lemma}
  Each $z_j$ lies on the interior of $\Omega$.
\end{lemma}
\begin{proof} By symmetry it suffices to consider $z_1$.
    Since $f$ is a twofold branched covering, and the map $f$ is conformal on the interior of the unit disc $\mathbb{D}$, we have $u_1\in \mathbb C\setminus \mathbb D$, and $z_1$ has only one preimage $u_1$. If $z_1$ was in $ \mathbb C\setminus \Omega$, then since we require $f$ to be surjective from $\mathbb D(0,1)$ to $\mathbb C\setminus \Omega $, we must have $u_1$ in the interior of the unit disc, which is a contradiction.
    
    \end{proof}

By definition, the deck transformation, denoted $\mathrm{deck}(u)$, is characterised by the condition that
\begin{equation} \label{aut of f deck}
    f(u)=f(\mathrm{deck}(u)).
\end{equation}
Since $f$ is given by \eqref{confm}, one can see that
\begin{equation} \label{deck rational}
    \mathrm{deck}(u)=\frac{ a-b}{ab }\frac{u}{u-1/b}+\frac{1}{a},
\end{equation}
and the fixed points of $\mathrm{deck}(u)$ are the critical points of $f$, $u_1$ and $u_2$.      
The map $f$ has two inverses denoted by $F_1$ and $F_2$  which are determined by their behaviour at $\infty$,
\begin{equation}\label{invasy}
    F_{1}(z)=\frac{\rho}{z}+\mathcal{O}(z^{-2} ),\qquad  F_{2}(z)=\frac{1}{a}+\frac{\rho(b-a)}{a}\frac{1}{z}+\mathcal{O}(z^{-2}).
\end{equation}
In the case of the first of these, $F_1$ maps the exterior of the droplet $\Omega$ to the unit disc $\mathbb D$ as already noted in the sentence above Theorem \ref{thm1}.
The two inverses are related by
\begin{equation} \label{F12 deck}
    F_{2}(z)=\mathrm{deck}(F_{1}(z)).
\end{equation}

These inverse functions have the explicit forms
\begin{equation}\label{defF1}
   F_1(z)= \frac{b \rho+z-\sqrt{(z-z_1)(z-z_2)}}{2 a z},
\end{equation}
\begin{equation}\label{defF2}
   F_2(z)= \frac{b \rho+z+\sqrt{(z-z_1)(z-z_2)}}{2 a z}.
\end{equation}
The branch cut of the square root is taken as a simple curve $\mathcal{B}$ which joins $z_1$ and $z_2$ while intersecting the real line in the interval $(-1/w,0)$ once.
Moreover, $F_{1}$ has analytic continuation to a second sheet by $F_2$ which is connected to the first sheet in a crisscross manner across $\mathcal{B}$. We will in particular make a precise choice of this cut which will turn out to be the mother body.

\begin{figure}
    \centering
    \includegraphics[scale=1,width=0.6\linewidth]{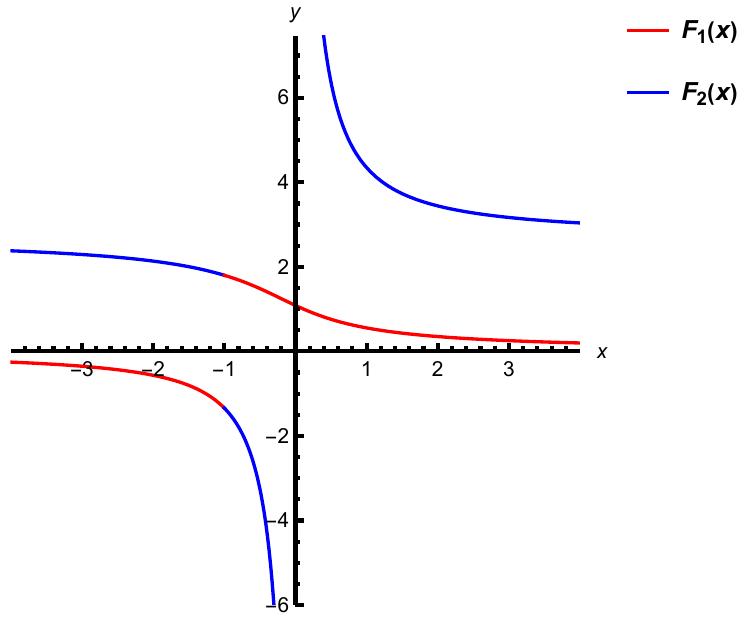}
    \caption{The plot of the function $F_1$ in red and its analytic continuation $F_2$ in blue on the real line. One observes $F_2$ has a pole at $0$.}
\end{figure}

\subsection{Spherical Schwarz function}
Denote 
\begin{equation} \label{def of nu}
\nu=\frac{1}{1+Q_0+Q_1}\nu_0,
\end{equation} 
where we recall the definition $\nu_0$ is the equilibrium measure of \eqref{enf}. 
Applying $\partial_z$ to the variational conditions \eqref{frost1}, we obtain
\begin{equation}\label{ctrans}
  \int \frac{d\nu(s)}{u-s}+\frac{Q_1}{1+Q_0+Q_1}\frac{1}{u-w}=\frac{\overline{u}}{1+|u|^2}, \qquad u\in  \Omega.
\end{equation}

Since $f$ maps the unit circle 
to a closed simple curve $\partial\Omega$, we have $ F_1: \partial\Omega \mapsto \mathbb S^1.$
Note that if 
\begin{equation}
    f(u)=z, \qquad |u|=1, \: z\in \partial \Omega,
\end{equation}
then 
\begin{equation}\label{2.12}
    f (1/F_1(z)) = f(1/u)= f(\overline{u})=\overline{f(u)}=\overline{z}.
\end{equation}
Here we have used that the rational map $f$ has real coefficients.

We define 
\begin{equation} \label{defS1}
S_1(z):= \frac{ f(1/F_1(z))}{1+z  f(1/F_1(z))}, \qquad z \in \mathbb{C} \setminus \mathcal{B}. 
\end{equation} 
It is immediate that $S_{1}(z)$ has a pole at $z=w$ with residue $Q_1/(1+Q_0+Q_1)$.
Also, 
it is meromorphic in $\mathbb C\setminus\mathcal{B}$.
From the precedent of \cite{CK19}, we call $S_1$ the \textit{spherical Schwarz function} of the droplet $\Omega$.

We know from \cite[Eq.~(2.48)]{BFL25} that with $u$ replaced by $f(u)$, $|u| \le 1$ and thus $f(u) \in \overline{\mathbb C \backslash \Omega}$, the RHS of (\ref{ctrans}) can be replaced by the meromorphic function $f(1/u)/(1 + f(u) f(1/u))$. 
Equivalently, we have 
\begin{equation} \label{Cauchy trans and S1}
 \int \frac{d\nu(s)}{z-s}+\frac{Q_1}{1+Q_0+Q_1}\frac{1}{z-w}=  S_{1}(z), \qquad z \in \overline{\mathbb C \backslash \Omega}. 
\end{equation} 
One notes from (\ref{2.12}) that the equality herein for $z \in \Omega$ is consistent with (\ref{ctrans}).  

We define the analytic continuation of $S_1$ across the branch cut   $\mathcal{B}$  (the precise form is yet to be chosen) joining $z_1,z_2$ by the map $F_2$ as 
\begin{equation}\label{defS2}
    S_{2}(z):=\frac{ f(1/F_2(z)) }{1+z f(1/F_2(z))}. 
\end{equation}
Then defining $\mathcal{R}$ to be a two-sheeted Riemann surface with sheets $\mathcal{R}_j$ ($j=1,2$), the map 
\begin{equation}
    S(z):=\begin{cases}
        S_1(z)\quad z \in \mathcal{R}_1,\\
        S_2(z)\quad z \in \mathcal{R}_2,
    \end{cases}
\end{equation}
defines a meromorphic function on $\mathcal{R}$; see Figure~\ref{Fig_Riemann surface two sheets}.

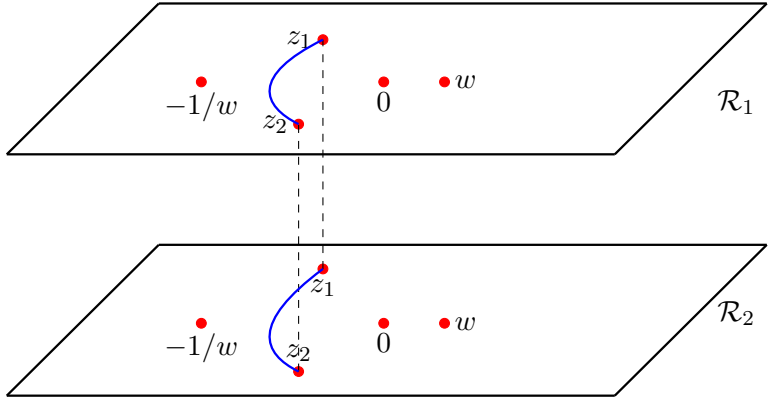
\begin{figure}
\hspace{2cm}
\begin{tikzpicture}[scale=0.8]
         \draw[line width=0.3mm] (0,0) -- (10,0);
         \draw[line width=0.3mm] (0,0)--(2.5,2.5);
           \draw[line width=0.3mm] (2.5,2.5)--(12.5,2.5);
           \draw[line width=0.3mm](10,0)--(12.5,2.5);
            \draw[line width=0.3mm] (0,4) -- (10,4);
         \draw[line width=0.3mm] (0,4)--(2.5,6.5);
           \draw[line width=0.3mm] (2.5,6.5)--(12.5,6.5);
           \draw[line width=0.3mm](10,4)--(12.5,6.5);

               \filldraw[red] (5.2,2.1) circle (2.3 pt)  node[below,black]{$z_1$};
              \filldraw[red] (4.8,0.4) circle (2.3 pt)  node[above,black]{$z_2$};
                \filldraw[red] (4.8,4.5) circle (2.3 pt)  node[left,black]{$z_2$};
                  \filldraw[red] (5.2,5.9) circle (2.3 pt)  node[left,black]{$z_1$};
          \filldraw[red] (6.2,1.2) circle (2.3 pt)  node[below,black]{$0$};
                 \filldraw[red] (7.2,1.2) circle (2.3 pt)  node[right,black]{$w$};
         \filldraw[red] (6.2,5.2) circle (2.3 pt)  node[below,black]{$0$};
                 \filldraw[red] (7.2,5.2) circle (2.3 pt)  node[right,black]{$w$};
                  \filldraw[red] (3.2,5.2) circle (2.3 pt)  node[below,black]{$-1/w$};
                        \filldraw[red] (3.2,1.2) circle (2.3 pt)  node[below,black]{$-1/w$};
        
             \draw[dashed](5.2,2.1)--(5.2,5.9); 
               \draw[dashed](4.8,0.4)--(4.8,4.5);
               \draw [line width=0.3mm, color=blue] (4.8,0.4)..controls(4.2,0.7)and(4,1.2)..(5.2,2.1);
                  \draw [line width=0.3mm, color=blue] (4.8,4.5)..controls(4.2,4.8)and(4,5.3)..(5.2,5.9);

                   \draw (12,1)  node[above]{$\mathcal{R}_2$};
                     \draw (12,4.5)  node[above]{$\mathcal{R}_1$};

   \end{tikzpicture}
    \caption{The Riemann surface $\mathcal{R}$}
    \label{Fig_Riemann surface two sheets}
\end{figure}
 
 We now determine its poles, zeros and an underlying algebraic equation. 
From the asymptotic behaviour given by \eqref{invasy}, it is easy to see that  
\begin{equation}\label{S1be}
    S_{1}(z)=\frac{b \rho^2}{z \left(a+b \rho^2\right)}+\mathcal{O}(z^{-2})=\frac{Q_1+1}{1+Q_0+Q_1}\frac{1}{z}+\mathcal{O} (z^{-2}), \qquad z \to \infty.
\end{equation}  
An alternative derivation is to recall that $S_{1}(z)$ has a pole at $z=w$  with residue $Q_1/(1+Q_0+Q_1)$, and to note that by \eqref{def of nu}, the Cauchy transform of the droplet measure satisfies 
$$\int \frac{d\nu(s)}{z 
-s}=\frac{1}{1+Q_0+Q_1}\frac{1}{z}+\mathcal{O}(z^{-2}), \qquad z \to \infty.$$
 In addition, by \eqref{invasy}, we have  
\begin{equation}\label{S2asy}
    S_{2}(z)=\frac{1}{z}+\mathcal{O}(z^{-2}), \qquad z \to \infty.
\end{equation}
Due to the fact $F_{2}(z)$ has a pole at 0 (as it should as $f$ maps $\infty$ to zero), we have by residue calculation that 
\begin{equation}\label{pole}
    S_{2}(z)=\frac{Q_1+1}{1+Q_0+Q_1}\frac{1}{z}+\mathcal{O}(z^{-2}) , \qquad z \to 0. 
\end{equation}
\begin{remark}
Note that $f$ has a pole at $1/a$ and therefore 
$1/F(f(a))=1/a$ at $f(a)$, telling us that $S$ has an isolated singularity at $f(a)$. However, due to the pole appearing in the numerator and denominator together in \eqref{defS2}, the singularity is removable.
\end{remark}

Recall \eqref{speq1}. 
Since the sum of residues of a meromorphic function $S_2$ (viewed on the Riemann surface $\mathcal{R}$) should vanish, it follows from \eqref{defS2}, \eqref{pole}, \eqref{S1be}, \eqref{S2asy} and \eqref{speq1} that $S_{2}$ has a pole at $-1/w$ with residue $(1+Q_0)/(1+Q_0+Q_1)$. In addition, by construction \eqref{defS1} and \eqref{defS2}, we have a zero at $f(b)$ on the first sheet. To summarise, we have the following picture for zeros and poles of $S$ on the Riemann surface $\mathcal{R}$.
\begin{itemize}
    \item Poles: At $w$ on the first sheet $\mathcal{R}_1$ with residue  $\frac{Q_1}{1+Q_0+Q_1}$. At $0,-1/w$ on the second sheet $\mathcal{R}_2$ with residues $\frac{1+Q_1}{1+Q_0+Q_1}$ and $\frac{1+Q_0}{1+Q_0+Q_1}$ respectively.
    \item Zeros : $f(b),\infty$ at the first sheet $\mathcal{R}_1$ and $\infty$ on the second sheet $\mathcal{R}_2$. 
\end{itemize} 

 \begin{lemma}\label{node}
     There exists $c_0$, with $c_0>w$, such that $S_1(c_0)=S_{2}(c_0)$.
 \end{lemma}
 \begin{proof}
     We first observe that $S_1,S_2$ are real and analytic in $(w,\infty)$. In addition, $S_1$ has a pole at $w$, with a positive residue, giving us $S_{1}(s)>S_{2}(s)$ for $s\in (w, w+\varepsilon)$, with $\varepsilon>0$ sufficiently small. Next, from \eqref{S1be} and \eqref{S2asy} we can conclude $S_{2}(s)>S_1(s)$ for sufficiently large $s$. The proof is then complete, invoking the intermediate value theorem.
 \end{proof}

\subsection{Spectral curve}
With the knowledge in the previous subsection, we are ready to determine the underlying algebraic curve, referred to as the \textit{spectral curve}.
The symmetric functions of the branches of a meromorphic function are rational; hence, we have $S_{1}+S_{2}$ and $S_1S_2$ are rational functions.
Define,
\begin{align}
\begin{split} \label{P1}
P_{1}(z) & :=    S_{1}(z)+S_{2}(z)
\\
& =\frac{1+Q_0}{1+Q_0+Q_1}\frac{1}{z+1/w} +\frac{Q_1}{1+Q_0+Q_1}\frac{1}{z-w} +\frac{Q_1+1}{1+Q_0+Q_1}\frac{1}{z}, 
\end{split}
\end{align} 
and 
\begin{equation}\label{P2}
  P_{2}(z) :=  S_{1}(z)S_{2}(z)
  = \frac{Q_1+1}{1+Q_0+Q_1}\frac{z-f(b)}{(z-w)(z+1/w)z}.
\end{equation}
Here, to evaluate the coefficients in \eqref{P2}, we have used the asymptotic behaviours \eqref{S1be} and \eqref{S2asy}, as well as the fact that $S$ has a zero at $f(b)$ by construction \eqref{defS1}, \eqref{defS2}, \eqref{2.12} and \eqref{confm}.

Hence we have determined the \textit{spectral curve} completely, as specified by
\begin{equation}\label{spectralcurve}
   P(S,z)= S^2+P_1(z)S+P_{2}(z)=0,
\end{equation}
\begin{figure}
    \centering
    \includegraphics[scale=0.8,width=0.5\linewidth]{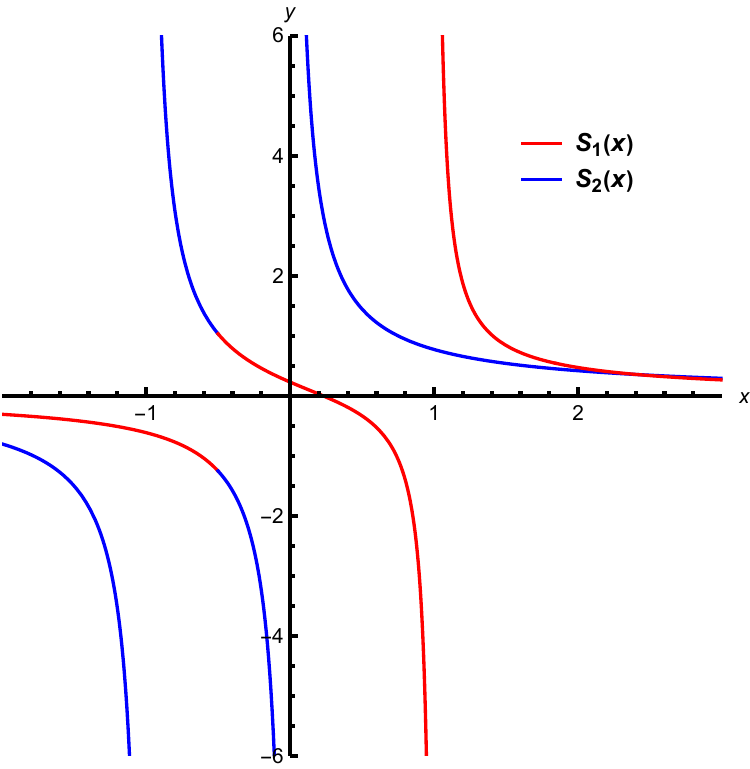}
    \caption{The Schwarz function $S_1$ and $S_2$ on the real axis. One can observe three poles, one zero and a node at the real axis.}
    \label{Schwarzplotl}
\end{figure}
where $P_{1}(z)$ and $P_{2}(z)$ are given by the RHS of \eqref{P1} and $\eqref{P2}$.
\begin{remark}
   $P(S,z)$ is not symmetric in $S,z$. In $S$ it has degree 2, while in $z$ degree 3.
\end{remark}

We now write the rational function
\begin{equation} \label{def of R}
R(z):= P_1(z)^2-4 P_2(z) 
\end{equation}
for the discriminant of the algebraic curve \eqref{spectralcurve}.
Completing the square in \eqref{spectralcurve} we have 
\begin{equation}\label{S2id}
    \Big(S-\frac{P_1(z)}{2}\Big)^2
    =\frac{1}{4}R(z).
\end{equation}  
By (\ref{S2id}), we obtain 
\begin{equation} \label{S2 1 sqrtR}
S_2(z)-\frac{P_{1}(z)}{2}=\frac{1}{2}\sqrt{R(z)},\qquad   S_1(z)-\frac{P_{1}(z)}{2}=-\frac{1}{2}\sqrt{R(z)}. 
\end{equation}
In the above and what follows the branch cut of $\sqrt{R(z)}$ is taken to be $\mathcal{B}$ (the same branch cut of $F_1$ defined in \eqref{defF1}), and we also have for this branch that as $z\rightarrow\infty$,
\begin{equation} \label{R infty asymp}
    \sqrt{R(z)}=\frac{Q_0}{1+Q_0+Q_1}\frac{1}{z}+\mathcal{O}(z^{-2});
\end{equation}
see \eqref{S1be}, \eqref{S2asy} and \eqref{P1}.
In addition, it follows from \eqref{S2 1 sqrtR} that 
\begin{equation} \label{R S12 square}
R(z)=   ( S_{1}(z)-S_{2}(z))^2. 
\end{equation}

We can make $R(z)$ more explicit. By using \eqref{S1be}, \eqref{S2asy}, \eqref{P1}, \eqref{P2} and \eqref{def of R}, we have  
\begin{equation}
    R(z)=\Big(\frac{Q_0}{1+Q_0+Q_1}\Big)^2\frac{1}{z^2}+\mathcal{O}(z^{-3}), \quad z\rightarrow \infty.
\end{equation}
Also, we have $R(z)$ to be a rational function with double poles at $w,-1/w,0$.
Hence, it must be that 
\begin{equation}
    R(z)=\frac{P_4(z)}{(z-w)^2(z+1/w)^2z^2},
\end{equation}
where $P_{4}(z)$ is a quartic polynomial. We determine its zeros. As we observed that $R(z)$ coincides with the discriminant of the spectral curve, the zeros occur when $S_1(z)=S_2(z)$. Two solutions are evidently the branch points $z_1,z_2$. Further Lemma \ref{node} gives us $c_0$ is a double zero. Thus we have
\begin{equation}\label{defR}
    R(z)= \Big(\frac{Q_0}{1+Q_0+Q_1}\Big)^2\frac{(z-z_1)(z-z_2)(z-c_0)^2}{(z-w)^2(z+1/w)^2z^2},
\end{equation}
where $z_1, z_2$ are given as the critical values of the conformal map $f$ and  $c_0$ is a multiple zero of  $\mathrm{Disc}(P(S,z),S)$.
In regards to this, we also have from \eqref{S2 1 sqrtR} that
\begin{equation}
R(z)=    \Big(S_2(z)-\frac{P_{1}(z)}{2}\Big)^2
\end{equation}
and $S_2(z)$ has a pole at $0$ with residue $(Q_1+1)/(Q_0+Q_1+1)$. 
By comparing the leading-order coefficient of $1/z^2$ in the two expressions for $R$ as $ z \to 0 $, and using \eqref{pole}, \eqref{P1}, along with the relation $ z_2 = \overline{z}_1 $, we obtain
\begin{equation} \label{c0 evaluation}
    c_0=\frac{1+Q_1}{Q_0 }\frac{1}{|z_1|}.
\end{equation}

Next we make note of a particular property relating to $c_0$ which will be required in subsequent working.
\begin{lemma}\label{crit1}
    Suppose $z\in \mathbb C \setminus \partial \Omega$ is such that 
    \begin{equation} \label{2.32}
    S_{1}(z)=\frac{\overline{z}}{1+|z|^2}.
    \end{equation}
    Then $z=c_0$.
\end{lemma}

\begin{proof}
    Let $z=f(v)$. Then by \eqref{defS1} the condition \eqref{2.32} is rewritten as
\begin{equation}
   \frac{f(1/v)}{1+f(v)f(1/v)}=\frac{f(\overline{v})}{1+f(v)f(\overline{v})}.
\end{equation}
Simplifying shows 
\begin{equation}
    f(1/v)=f(\overline{v}),
\end{equation}
which means 
\begin{equation}\label{deck1}
    \mathrm{deck}(v)=1/\overline{v},
\end{equation}
where we have used \eqref{aut of f deck}.
Therefore we have 
\begin{equation}
    S_{1}(z)=\frac{f(1/v)}{1+f(1/v)f(v)}=\frac{f(\overline{v})}{1+f(v)f(\overline{v})}=\frac{f(1/\mathrm{deck}(v))}{1+f(1/\mathrm{deck}(v))f(v)}=S_{2}(z),
\end{equation}
which implies $z=z_1$ or $z_2$ or $c_0$ by the discussion above \eqref{defR}.
Since $u_1$ and $u_2$ are fixed points of $\mathrm{deck}(u)$, it follows that if $z=z_k$ for $k=1$ or $2$, by \eqref{deck1}, $|u_k|=1$, which implies that $z \in \partial \Omega$.
Hence, $z$ must be $c_0$.
\end{proof}

Recall \eqref{frost1} for the definition of $\mathcal{U}_{\mathrm{2D}}$. As a corollary, we observe that $c_0$ is the only critical point of $\mathcal{U}_{\mathrm{2D}}$ in $\mathbb C\setminus\Omega$.
\begin{lemma}\label{strictineq}
    The inequality in \eqref{frost1} is strict. That is
    \begin{equation}
           \mathcal{U}_{2D}(z)> 0, \quad z\in \mathbb C\setminus \Omega.
    \end{equation}
\end{lemma}
\begin{proof}
By \eqref{frost1} we find $\partial \mathcal{U}_{2D}=\overline{\partial} \mathcal{U}_{2D}=0$ on $\Omega$. Assume that there is a local minimum of $ \mathcal{U}_{2D}$ on  $\mathbb C\setminus \Omega$. Then the Mountain Pass Theorem guarantees that there is a critical point on a path connecting $\Omega$ and the local minimum. 
By \eqref{Cauchy trans and S1}, one can notice that $\partial \mathcal{U}_{2D}$ is equivalent to the condition \eqref{2.32}.
On the other hand, by Lemma \ref{crit1} there is at most one critical point in $ \mathcal{U}_{2D}$, therefore there is no local minimum of $ \mathcal{U}_{2D}$ on  $\mathbb C\setminus \Omega$.
\end{proof}

\section{The mother body}\label{mbsection}

In this section Theorem \ref{thmmoth} will be established.

\subsection{Analysis of the trajectories of the quadratic differential}

We are interested in the trajectories of the quadratic differential $$R(z)\,dz^2<0.$$
This in essence will give us a real measure $\mu$ with density
$$d\mu(z)=\frac{1}{2\pi i}\sqrt{R(z)}\, dz,$$
supported on the critical trajectories. Here, $dz$ represents the complex line element.

We recall from \eqref{defR} and \eqref{c0 evaluation}
\begin{equation}
    R(z)=\Big(\frac{Q_0}{1+Q_0+Q_1}\Big)^2\frac{(z-z_1)(z-z_2)(z-c_0)^2}{z^2(z-w)^2(z+1/w)^2}, \qquad c_0=\frac{1+Q_1}{Q_0 }\frac{1}{|z_1|}. 
\end{equation}
In that regard, we have the meromorphic differential 
$$
(S_1(z)-S_{2}(z))^2\, dz^2=R(z)\, dz^2.
$$
By general theory \cite{J58, P75, S84}, we have the following rules for the trajectories of the quadratic differential $R(z)\, dz^2<0$.
\begin{itemize}
    \item[(1)] $z_1,z_2$ are simple zeros of $R(z)$. Hence, there are three equiangular arcs emanating from $z_1$ and $z_2$.
    \item[(2)] $c_0$ is a double zero of $R(z)$. Therefore, there are four equiangular arcs emanating from $c_0$.
    \item[(3)] $0,w,-1/w$ are double poles with positive residues. Hence, the trajectories near $0,w,-1/w$ are locally circular.
    \item[(4)] As $z\rightarrow\infty$ we have $R(z)\sim c_1z^{-2}$, with $c_1>0$ as a result, infinity is a double pole with positive residue. Then the trajectories near infinity are also circular, that is closed loops.
    \item [(5)] We use the Teichmüller's Lemma \cite{P75, S84}. A consequence of it is that a simply connected domain bounded by critical trajectories that does not contain a pole on its boundary must have a pole in its interior. 
    \item[(6)] By symmetry, if $\gamma$ is a trajectory, then so is $\overline{\gamma}$.
\end{itemize}

Due to item (4), the short critical trajectories are finite and do not escape to infinity. Consequently, the trajectories emanating from $ z_1 $ in the upper half-plane must terminate at either \( z_2 \) or \( c_0 \). Moreover, at most one of these three trajectories can end at \( c_0 \), as ensured by items (5) and (6).

In the following, we show all the critical trajectories emanating from $z_1$ terminate at $z_2$.

\subsubsection{Potential theoretic preliminaries}

We begin with some definitions.


\begin{definition} For $z \in \mathbb{C}\setminus \mathcal{ B}$, define
    \begin{equation}\label{defU}
    \mathcal{U}(z):=\log(1+|z|^2)-\log(1+|z_0|^2)-2\Re \int_{z_0}^{z}S_{1}(s) \, ds,
\end{equation}
where the contour of integration is away from the branch cut $\mathcal{B}$ and $z_0=f(1)$ is in $\partial \Omega$.
\end{definition}

Notice that by \eqref{frost1} \eqref{Cauchy trans and S1}, \eqref{def of nu} and \eqref{defU}
we have 
$$
\frac{\partial\mathcal{U}_{\mathrm{2D}}}{\partial z}=(1+Q_0+Q_1)\frac{\partial\mathcal{U}}{\partial z}, \qquad z\in \mathbb  C\setminus \Omega,
$$
while $$\mathcal{U}_{\mathrm{2D}}(z_0)=(1+Q_0+Q_1)\mathcal{U}(z_0)=0.$$
Hence we can conclude
\begin{equation}\label{relaU}
\mathcal{U}_{2D}(z) = (1+Q_0+Q_1) \mathcal{U}(z), \qquad z \in \overline{\mathbb C \setminus \Omega}. 
\end{equation}
Next, we define an  integral with $\sqrt{R(s)}$ as the integrand. 
\begin{definition}
  For $z \in \mathbb{C}\setminus \mathcal{ B}$, define
\begin{equation}\label{defU0}
      \mathcal{U}_{0}(z):=\Re \int_{z_1}^{z} \sqrt{R(s)} \, ds,
\end{equation}
where the contour of integration is in $\mathbb C \setminus\mathcal{B} $.
\end{definition}

In \eqref{defU} and \eqref{defU0}, the contour is chosen away from the poles at $0,w$ and $-1/w$. Since the residues are purely real the real part of the integral then does not depend on the choice of the contour. Recall that $c_0$ is defined in Lemma~\ref{node} and is given by the expression in \eqref{c0 evaluation}.

\begin{proposition}\label{traj1}
We have 
$$
\mathcal{U}_{0}(c_0)=\mathcal{U}(c_0)>0.
$$
\end{proposition}

\begin{proof}
    The inequality is an easy consequence of Lemma \ref{strictineq} and \eqref{relaU} combining with Lemma \ref{node} which tells us $c_0\in \mathbb C \setminus \Omega$.
    We prove the first equality. Let us denote $F_{1}(c_0)=v$, where $F_1$ is the conformal map defined \eqref{defF1}. Then by \eqref{deck1} we have $\mathrm{deck}(1/v)=\overline{v}$. Then by \eqref{R S12 square}, \eqref{defS1}, \eqref{defS2} and \eqref{F12 deck}, we have 
    \begin{align} \label{real1}
    \begin{split}
    &\quad \int_{z_1}^{c_0}\sqrt{R(s)}\, ds= \int_{z_1}^{c_0}(S_{2}(s)-S_{1}(s)) \, ds
    \\
    &=\int_{u_1}^{v}\frac{f(1/\mathrm{deck}(u))}{1+f(1/\mathrm{deck}(u))f(u)} d(f(\mathrm{deck}(u)))  -\int_{u_1}^{v}\frac{f(1/u)}{1+f(1/u)f(u)}d(f(u))
    \\&=\int_{v}^{\mathrm{deck}(v)}\frac{f(1/u)}{1+f(1/u)f(u)}d(f(u)). 
    \end{split}
    \end{align} 
    Recall $z_0\in \partial \Omega$ is given by $f(1)$.
  We decompose the last integral as $I_1+I_2$, where  
    $$
    I_1:= \int _{v}^{1}   \frac{f(1/u)}{1+f(1/u)f(u)}d(f(u)), \qquad I_2:= \int _{1}^{\mathrm{deck}(v)}\frac{f(1/u)}{1+f(1/u)f(u)}d(f(u)). 
    $$
   
    Notice that by definition \eqref{defS1}, we have 
    \begin{equation}\label{pot1}
        I_1=-\int_{z_0}^{c_0}S_{1}(s)\,ds. 
    \end{equation}
    On the other hand, we have with $u=1/x$, 
    \begin{align}
    \begin{split}
I_2 
    &=\int_{1}^{\mathrm{deck}(v)}d\log (1+f(1/u)f(u))+\frac{f(u)f'(1/u)}{u^2(1+f(1/u)f(u) ) } \, du
    \\
    &=\Big[\log (1+f(1/u)f(u))\Big]_{1}^{\mathrm{deck}(v)}+\int_{1/\mathrm{deck}(v)}^{1}\frac{f(1/x)f'(x)}{1+f(1/x)f(x)}\,dx   
    \\
    &=\log(1+|c_0|^2)-\log(1+|z_0|^2)-\int^ {c_0}_{z_0} S_{1}(x)\,dx,\label{pot2}
    \end{split}
    \end{align}
where in the last equality of \eqref{pot2} we used \eqref{defS1}. Combining \eqref{pot1}, \eqref{pot2} and \eqref{defU} we obtain 
\begin{equation}\label{steas}
    \begin{aligned}
    &\quad \mathcal{U}_{0}(c_0)=\Re \int_{z_1}^{c_0} \sqrt{R(s)} \, ds=\Re ( I_1+ I_2)\\
    &= \log(1+|c_0|^2)-\log(1+|z_0|^2)-2\Re\int^{\overline{c}_0}_{\overline{z}_0} S_{1}(x)\, dx
    \\
    &= \mathcal{U}(c_0).
    \end{aligned} \end{equation}
This finishes the proof.
\end{proof}

\begin{remark}
We expect Proposition \ref{traj1} to be quite general. We point out that a  variant of this proposition is given in \cite[Appendix C]{BBLM15}.
\end{remark}

We have seen that the three equiangular trajectories emanating from \( z_1 \) must all terminate at \( z_2 \), forming two simply connected domains, \( \widetilde{\Omega}_1 \) and \( \widetilde{\Omega}_2 \), each containing at least one pole.  
By Teichmüller's Lemma \cite{P75,S84}, if a domain \( \widetilde{\Omega}_j \) contains two poles, it must also contain two zeros (counted with multiplicty, hence in this case it must contain $c_0$). This reduces to two cases:

\begin{figure}
    \centering
    \begin{tikzpicture}
    \node[anchor=south west,inner sep=0] at (0,0) 
    {\includegraphics[width=0.9\linewidth]{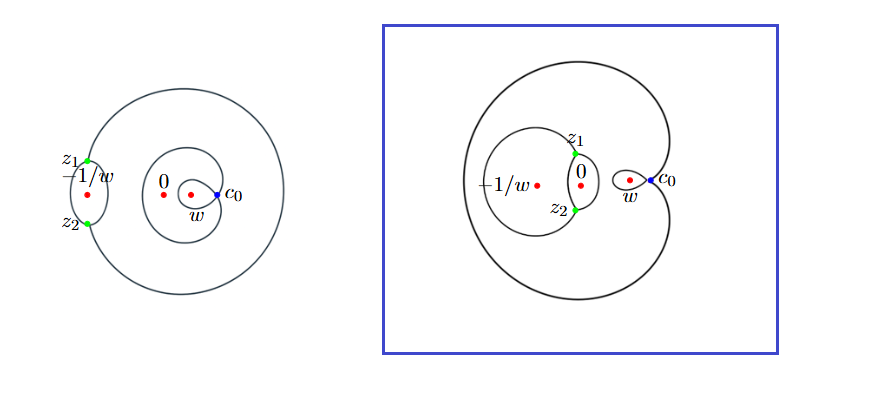}};
    \end{tikzpicture}
    \caption{Possible topological configuration of critical trajectories. Case I in left, Case II (in box) in right which satisfies the inequality in Proposition \ref{traj1}. }
    \label{topconfg}
\end{figure}

\begin{itemize}
    \item Case I: $\widetilde{\Omega}_1$ contains $0,w$ and $c_0$, while $\widetilde{\Omega}_2$ contains $-1/w$. This corresponds  left panel  of Figure \ref{topconfg}.
    \item Case II: $\widetilde{\Omega}_1$ contains $0$ but not $w$ and $c_0$ and $\widetilde{\Omega}_2$ contains $-1/w$. This corresponds to right panel of Figure \ref{topconfg}.
\end{itemize}

We now negate Case I. Since we have $\sqrt{R(x)}>0$ for $x>c_0$ and  $\sqrt{R(x)}<0$  for $x\in (w,c_0)$. Then $\int_{x}^{c_0}\sqrt{R(s)}ds<  0$ for $x\in (w,\infty) $ and $x\neq c_0$. Suppose Case I is taking place; then we have a trajectory (part of the boundary of $\Omega_1$) emanating from $z_1$ intersecting the real line at $x_0$ with $x_0>c_0$. Then one observes
\begin{equation}
  \Re  \int_{z_1}^{c_0} \sqrt{R(s)} \, ds= \Re  \int_{z_1}^{x_0} \sqrt{R(s)} \, ds+  \Re \int^{c_0}_{x_0} \sqrt{R(s)}\, ds<0.
\end{equation}
This contradicts Proposition \ref{traj1}.

\medskip 

As a corollary, we obtain the following.

\begin{corollary}\label{crittraj} $ $
  \begin{itemize}
      \item [(1)] All the three equiangular critical trajectories emanating from $z_1$ terminate at $z_2$.  We denote them by left, middle and right trajectory. The left trajectory intersects the real axis to the left of $-1/w$. The middle trajectory intersects the real axis in the interval $(-1/w,0)$ and the right axis intersects the real axis to the right of $0$. 
     These three critical trajectories form two bounded simply connected domains, $\widetilde{\Omega}_1$ (containing $0$) and  $\widetilde{\Omega}_2$ (containing $-1/w$). 
      \item [(2)] From $c_0$, emanate four equiangular critical trajectories, forming two loops, one enclosing $w$ and one enclosing all the poles $-1/w,0,w$.
  \end{itemize}
\end{corollary}
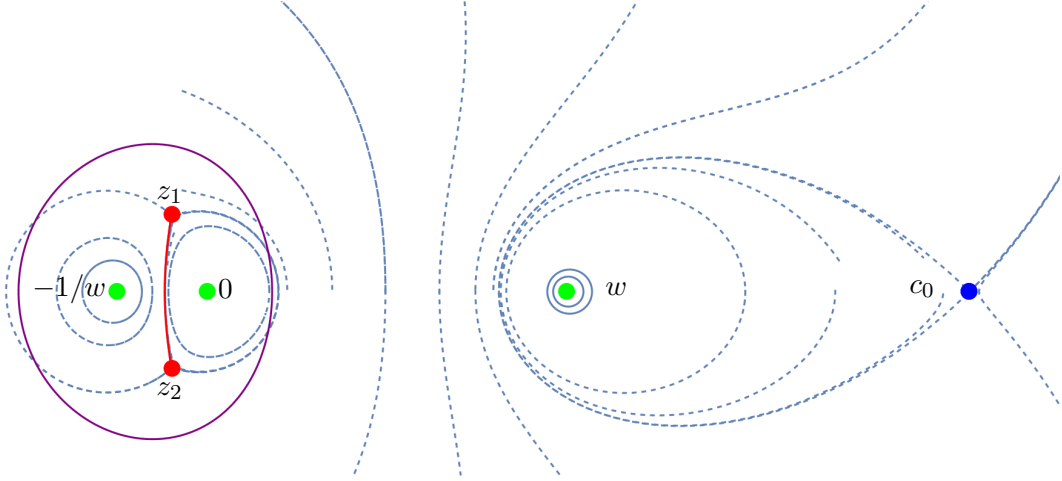
\begin{figure}
    \centering
    \begin{tikzpicture}
    \node[anchor=south west,inner sep=0] at (0,0) 
    {\includegraphics[width=1\linewidth]{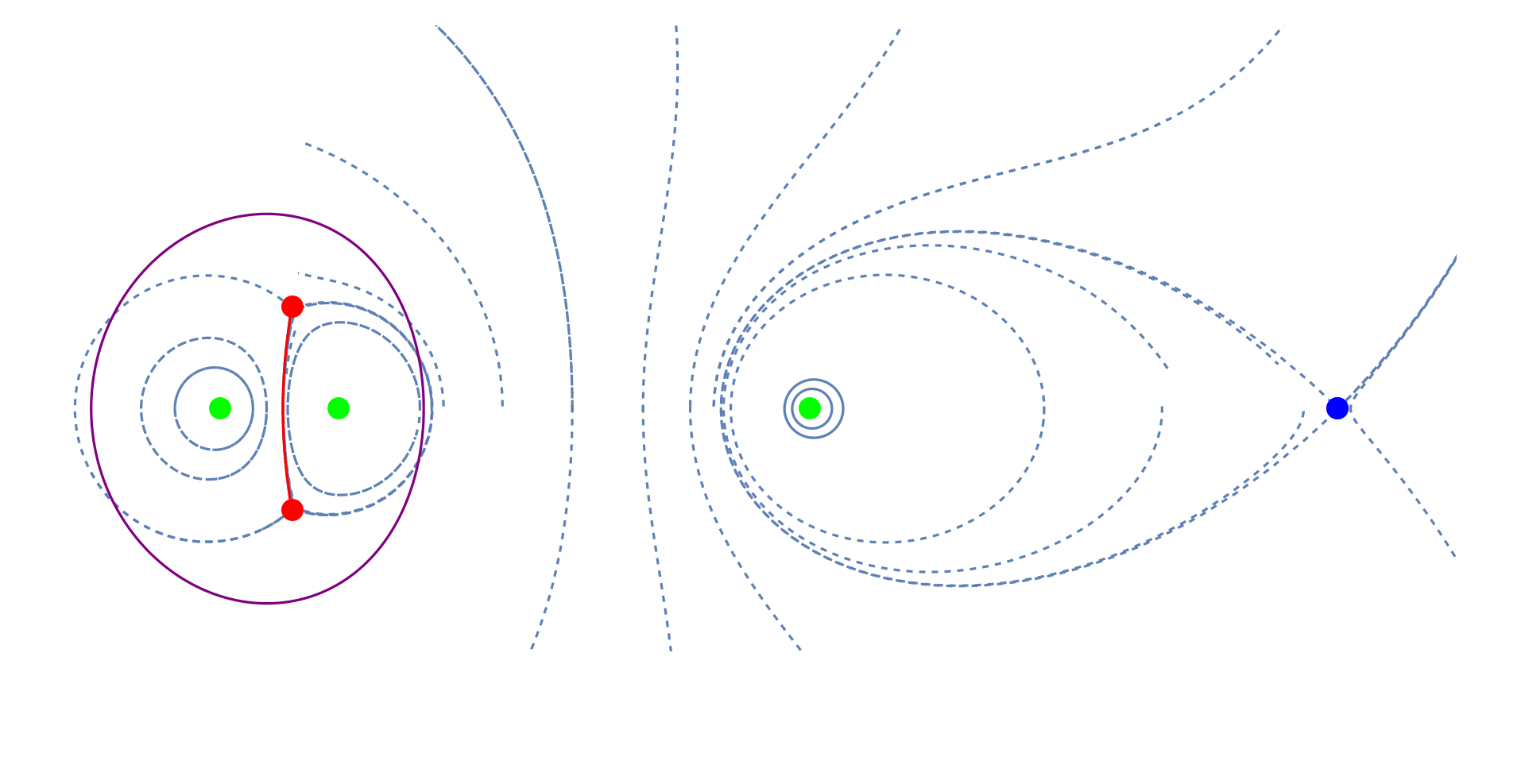}};
    \draw (2.9,2.7)  node[below]{$z_2$};
        \draw (2.9,5.3)  node[below]{$z_1$};
             \draw (2.2,3.8)  node[left]{$-1/w$};
                 \draw (3.4,3.8)  node[right]{$0$};
                                 \draw (8.5,3.8)  node[right]{$w$};
                                 \draw (12.5,3.8)  node[right]{$c_0$};
    \end{tikzpicture}
    \caption{Global structure of trajectories of the quadratic differential $(S_{1}(z)-S_2(z)^2\,dz^2$: mother body in red, droplet in purple (which is not a trajectory), poles in green, node in blue dots.}
    \label{globtraj}
\end{figure}
This finishes the discussion on the trajectories of the quadratic differential and confirms the numerical simulation in Figure \ref{globtraj}.
\begin{lemma}\label{orthtraj}
    The steepest ascent path $\Gamma_1$  of $\mathcal{U}_{0}$ from $z_1$ terminates at $c_0$ and similarly the steepest ascent path  $\Gamma_2$  of $\mathcal{U}_{0}$ from $z_2$ terminates at $c_0$. Moreover, we have $\mathcal{U}_{0}(z)>0$ for $z\in \Gamma_j$, where $j=1,2$.
    \end{lemma}

\begin{proof}
The proof follows along the same lines as that of \cite[Lemma~3.3, Appendix~C]{BBLM15}. That is, if $v=F_1(z)$ satisfies $\mathrm{deck}(v)=1/\overline{v}$, then LHS of \eqref{real1} is purely real. Since $c_0$ is such a point, we have $\int _{z_1}^{c_0}\sqrt{R(s)}\, ds=\mathcal{U}(c_0)>0$. Thus, $c_0$ lies on the steepest ascent path from $z_1$.  By symmetry we can conclude the same about $\Gamma_2$. See also Figure \ref{steep} for illustration.
\end{proof}

\subsection{Proof of Theorem~\ref{thmmoth}}
\begin{definition}\label{defmoth}
     Define $\Gamma_0$ as the unique trajectory of the quadratic differential $R(z)\,dz^2<0$ emanating from $z_1$ and terminating at $z_2$ which intersects the real line between $-1/w$ and $0$ which is guaranteed by Corollary \ref{crittraj}. Further, define     
\begin{equation}\label{mothn}
   d\mu(s):=\frac{1}{2\pi i}\sqrt{R(s)}\, ds, \quad s\in\Gamma_0 .
\end{equation}
\end{definition}

We are now ready to prove Theorem~\ref{thmmoth}. 
 
\begin{proof}[Proof of Theorem~\ref{thmmoth}]
According to Definition~\ref{defmoth}, \( \mu \) is a real measure supported on \( \Gamma_0 \), with a density that vanishes as a square root at \( z_1 \) and \( z_2 \), while remaining strictly positive in the interior.

\begin{figure}
    \centering
    \includegraphics[scale=0.6]{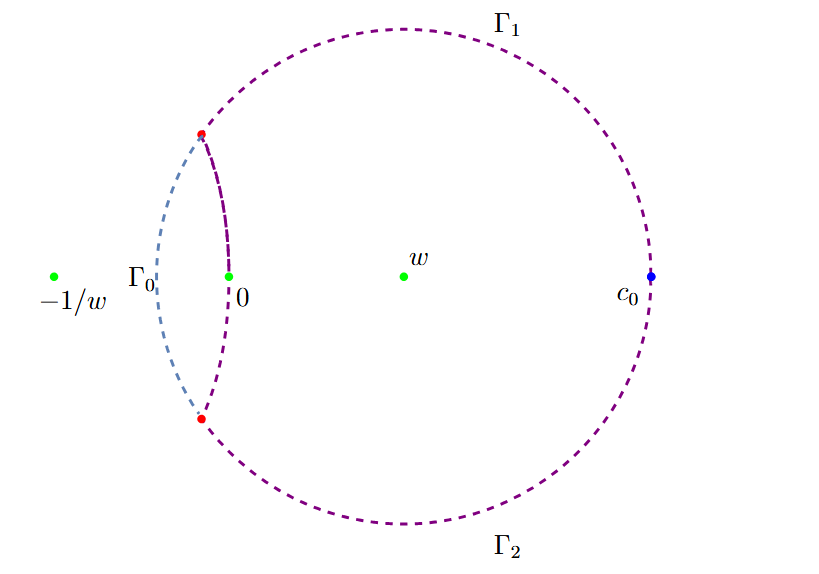}
    \caption{The contour $\Gamma=\Gamma_0\cup\Gamma_1\cup\Gamma_2$. The steepest ascent path $\Gamma_1$ and $\Gamma_2$ in dashed purple. $\Gamma_0$ in dashed blue.}
    \label{steep}
\end{figure}

Now choosing the branch cut $\mathcal{B}$ to be exactly on $\Gamma_0$ we obtain 
\begin{equation}
    \frac{1}{2\pi i}\left(S_{1,-}(z)-S_{1,+}(z)\right)dz=d\mu(z), \quad z\in \Gamma_0.
\end{equation} 
Recall the only pole of $S_1$ is at $w$  with residue $Q_1/(1+Q_0+Q_1)$ as given in \eqref{Cauchy trans and S1}. Then we obtain by the residue theorem,
\begin{equation}
    \int \frac{d\mu(s)}{z-s}=  \frac{1}{2\pi i} \int_{\Gamma_0} \frac{S_{1,-}(s)-S_{1,+}(s)}{z-s}\, ds=S_{1}(z)-\frac{Q_1}{1+Q_0+Q_1}\frac{1}{z-w}.
\end{equation}
A similar calculation shows
\begin{equation}\label{moth1}
    \begin{aligned}
        & S_{2}(z)=-\int \frac{d\mu(s)}{z-s}+\frac{1+Q_1}{1+Q_0+Q_1}\frac{1}{z}+\frac{1+Q_0}{1+Q_0+Q_1}\frac{1}{z+1/w}.
    \end{aligned}
\end{equation}
By comparing the behaviour at $\infty$ in the LHS and RHS of \eqref{moth1} using \eqref{S1be} and \eqref{S2asy}, we obtain 
\begin{equation}\label{tmass1}
    \mu(\mathbb C)=\frac{1}{1+Q_0+Q_1}.
\end{equation}
Since the density of $\mu$ which is given by $\sqrt{R(s)}ds$ is non-vanishing in the interior of $\Gamma_0$ and the total mass is positive, we then conclude $\mu$ is a positive measure.

In what follows we will denote,
\begin{equation} \label{def of Cauchy trans}
C^{\mu}(z)=\int\frac{d\mu(s)}{z-s}, \qquad z\in \mathbb C \setminus \mathrm{supp}(\mu). 
\end{equation}
Putting \eqref{moth1} and \eqref{P1} in \eqref{S2id},
we obtain
\begin{equation}\label{c1}
\bigg(C^{\mu}(z)+\frac{1}{2}\Big(\frac{Q_1}{1+Q_0+Q_1}\frac{1}{z-w}-\frac{1+Q_0}{1+Q_0+Q_1}\frac{1}{z+1/w}-\frac{Q_1+1}{1+Q_0+Q_1}\frac{1}{z}\Big)\bigg)^2=\frac{1}{4}R(z).
\end{equation}
Also, it is convenient to define the  scaling of $\mu$ and $R$. 
Thus for $\mu$ defined in \eqref{defmoth} we introduce the scaled form
\begin{equation}
    \mu_0=(1+Q_0+Q_1)\mu.
\end{equation}
Then due to \eqref{tmass1} we have $\mu_0$ to be a probability measure supported on $\Gamma_0$. That is 
\begin{equation}\label{tmass2}
    \mu_0(\mathbb C)=1.
\end{equation}
In relation to $R$ we define the scaled form
\begin{equation}\label{defR0}
    R_0(z)=(1+Q_0+Q_1)^2 R(z),
\end{equation}
and we introduce too
\begin{equation} \label{def of widetilde U U0}
    \widetilde{\mathcal{U}}(z)=(1+Q_0+Q_1)\mathcal{U}(z),\qquad \widetilde{\mathcal{U}}_0(z)=(1+Q_0+Q_1)\mathcal{U}_0(z).
\end{equation}  
This then establishes item (1) in Theorem \ref{thmmoth}, where we have also used \eqref{R infty asymp}.

\medskip 

Next, we verify item (2). For this, recall the definition $\mathcal{V}$ from definition \eqref{defV},
where the branch of the logarithm is chosen on $(-\infty,w]$. Then \eqref{c1} can be rewritten as 
\begin{equation}\label{c2}
    \Big(C^{\mu_0}(z)-\frac{1}{2}\mathcal{V}'(z)\Big)^2=\frac{1}{4}R_0(z).
\end{equation}
Recall $\Gamma_1$ and $\Gamma_2$ from Lemma \ref{orthtraj}. Now we are ready to define the contour 
\begin{equation} \label{def of Gamma 0 1 2}
\Gamma=\Gamma_0\cup\Gamma_1\cup \Gamma_2,
\end{equation}
which is a closed contour enclosing $[0,w]$ which also has $(\infty,-1/w]$ on the exterior (see Figure \ref{steep}).
From \eqref{c2} and Lemma \ref{orthtraj} integrating in $z$ and taking the real part, we obtain that there exists a real constant $\ell_0$ such that
\begin{equation}\label{moth2}
    2U^{\mu_0}(z)+\Re \mathcal{V}(z)+\ell_0\begin{cases}=0 \quad z\in \Gamma_0,\\
    >0  \quad z\in \Gamma\setminus  \Gamma_0.
    \end{cases}
\end{equation}
This establishes item (2) of Theorem \ref{thmmoth}.

\medskip 
For item (3), we first note that \( \mathcal{B} = \Gamma_0 \subset \Omega \). We do not provide a proof of this, as it is identical to that of \cite[Lemma~2.8]{BBLM15}. 
As a consequence, by combining  \eqref{def of nu}, \eqref{ctrans}, \eqref{tmass2},  and \eqref{defS1}, we obtain \eqref{moth3}. This establishes item (3) of Theorem~\ref{thmmoth}. 
\end{proof}

\begin{remark}
    It also follows from \eqref{c2} and \eqref{eqcon} that the contour $\Gamma_0$ is a \textit{S curve} (See \cite[~Lemma 5.4]{MR11}. That is we have,
    \begin{equation}
        \frac{\partial}{\partial n_+}\left( 2U^{\mu_0}(z)+\Re \mathcal{V}(z)\right)=  \frac{\partial}{\partial n_-} \left(2U^{\mu_0}(z)+\Re \mathcal{V}(z)\right),\quad z\in \Gamma_0,
    \end{equation}
    where $n_\pm$ are the unit normal vectors to $\Gamma_0$ at $z$. This \textit{S property} has appeared in several works such as \cite{KS15, R12,MR16}.
\end{remark}

\section{Contour orthogonality }\label{eqsection}
Crucial to our subsequent Riemann-Hilbert analysis is that the planar orthogonality relation with respect to the weight $e^{-N V(z)}$, with $V(z)$ given by (\ref{F10}), is equivalent to a non-Hermitian contour orthogonality with weight function (\ref{wC}).
We start by proving a general lemma that will be used to prove Proposition  \ref{contorth}.
\begin{lemma}\label{contorth0}
Let $j,k\leq N$. We have
  \begin{multline}\label{lemma1}
\int_{\mathbb C}z^{j}\overline{(z-w)}^{k}|z-w|^{2 N Q_1}  \Big( \frac{1}{1+|z|^2} \Big)^{N(1+Q_0+Q_1)+1} \,  dA(z)
         \\
         =\frac{G_{k,N}}{2i}     \oint_{\Gamma}z^j \frac{(z-w)^{NQ_1}}{(1+wz)^{N+NQ_{0}-k}} \frac{dz}{z^{k+1+NQ_{1}}},
\end{multline}
where \(\Gamma\) is a positively oriented simple closed contour that contains \([0, w]\) in its interior and \((-\infty, -1/w]\) in its exterior. Here, 
\begin{equation}\label{defbeta}
    G_{k,N}:=\frac{\Gamma(N +N Q_0-k) \Gamma(1+k+NQ_1)}{\, \Gamma(N(1+Q_0+Q_1) +1 )}.
\end{equation}
\end{lemma}

\begin{remark}
Before the proof, let us consider the special case $w=0$ of the formula \eqref{lemma1}. For $w=0$, the LHS of \eqref{lemma1} can be evaluated using polar coordinates and Euler's beta integral as  
\begin{align*}
\int_{\mathbb C}   \frac{z^{j}\overline{z}^{k}|z|^{2 N Q_1} }{( 1+|z|^2 )^{ N(1+Q_0+Q_1)+1 } }   \,  dA(z) &=  2 \pi \int_0^\infty \frac{ r^{2k+2NQ_1+1} }{ (1+r^2)^{ N(1+Q_0+Q_1)+1  } }\,dr \,  \delta_{j,k} =  2\pi i \, G_{k,N}\,  \delta_{j,k}. 
\end{align*}
On the other hand, the integral in the RHS of \eqref{lemma1} is evaluated as 
\begin{align*}
  \oint_{\Gamma} z^{j-k-1} \,dz = 2\pi i\,  \delta_{j,k}. 
\end{align*}
Thus one can directly observe that \eqref{lemma1} holds for $w=0$. 
\end{remark}

\begin{proof}[Proof of Lemma~\ref{contorth0}]
To shorten our notations we set
\begin{equation}\label{short}
\alpha=NQ_1,\quad \beta=N(1+Q_0+Q_1).\end{equation}
Then one  observes 
\begin{equation} \label{parcond}
	\alpha > 0, \qquad \beta >\alpha + N. 
\end{equation} 
We first consider the function 
\begin{equation} \label{fdef} g(z) = z^{j} (z-w)^\alpha
	\int_w^{\bar{z}} \frac{(s-w)^{\alpha + k}}{(1+zs)^{\beta + 1}} \, ds.  
\end{equation}
Since $\alpha$ and $\beta$ are not assumed to be integers, we will use principal branches for non-integers powers.
Thus $(z-w)^{\alpha}$ has the branch cut for $z \in (-\infty, w]$, $(s-w)^{\alpha +k}$ has the branch cut for $s \in (-\infty, w]$,  and given $z \in \mathbb C \setminus \{0\}$, $(1+zs)^{\beta+1}$ has the branch cut for $s \in [-1/z, -1/z \times \infty)$. 
Then it is immediate that \eqref{fdef} is defined on $\mathbb C \setminus (-\infty, w]$
and is $C^{\infty}$ there.

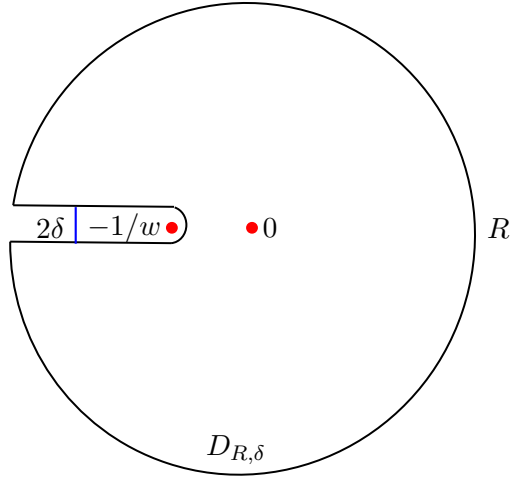
\begin{figure}
    \centering
   \tikzset{every picture/.style={line width=0.75pt}} 
\begin{tikzpicture}[x=0.75pt,y=0.75pt,scale=0.8]

\draw  [draw opacity=0] (179.12,196.3) .. controls (178.81,134.45) and (216.66,77.2) .. (277.03,57.04) .. controls (352.75,31.75) and (435.22,74.38) .. (461.23,152.26) .. controls (487.25,230.14) and (446.95,313.77) .. (371.23,339.07) .. controls (295.51,364.36) and (213.04,321.73) .. (187.02,243.85) .. controls (184.29,235.66) and (182.29,227.4) .. (180.98,219.16) -- (324.13,198.05) -- cycle ; \draw   (179.12,196.3) .. controls (178.81,134.45) and (216.66,77.2) .. (277.03,57.04) .. controls (352.75,31.75) and (435.22,74.38) .. (461.23,152.26) .. controls (487.25,230.14) and (446.95,313.77) .. (371.23,339.07) .. controls (295.51,364.36) and (213.04,321.73) .. (187.02,243.85) .. controls (184.29,235.66) and (182.29,227.4) .. (180.98,219.16) ;  
\draw    (179.12,196.3) -- (279.51,195.33) ;
\draw    (180.98,219.16) -- (281.37,218.19) ;
\draw  [draw opacity=0][line width=3] [line join = round][line cap = round] (317.33,214.27) .. controls (320.34,218.78) and (326.08,228.27) .. (331.33,228.27) ;
\draw  [draw opacity=0] (281.78,217.8) .. controls (285.99,216.37) and (289.06,211.98) .. (289.06,206.8) .. controls (289.06,200.72) and (284.85,195.74) .. (279.51,195.33) -- (278.79,206.8) -- cycle ; \draw   (281.78,217.8) .. controls (285.99,216.37) and (289.06,211.98) .. (289.06,206.8) .. controls (289.06,200.72) and (284.85,195.74) .. (279.51,195.33) ;

 \filldraw[red] (280,205) circle (2.3 pt)  node[left,black]{$-1/w$};
 
 \filldraw[red] (330,205) circle (2.3 pt)  node[right,black]{$0$};
   \draw (320,50)  node[above]{$ D_{R,\delta}$}; 
\filldraw[red] (470,205)  node[right,black]{$R$};
 \draw[-,blue, thick] (220,195) --++ (90:23);
 \draw[black](220,205) node[ below,left] {$2\delta$};

\end{tikzpicture}
 
    \caption{The domain where Stokes theorem is applied in Lemma \ref{contorth0}}.
    \label{contorthfig}
\end{figure}

Our first observation is that, $g$ has a continuous extension in $\mathbb C\setminus (-\infty,-1/w]$.

Let $z \to x \in (-1/w, w]$ with $\Im z > 0$. Then $(z-w)^{\alpha} \to |x-w|^{\alpha} e^{i\pi \alpha}$. The integral in \eqref{fdef} tends to an integral over the interval $[x,w]$. 
Since $1+xs >0$ for $-1/w < x < s < w$, the denominator of the integrand tends to $(1+xs)^{\beta+1}$. On the other hand, the numerator tends to $(-1)^k |s-w|^{\alpha + k} e^{-i\pi \alpha}$, since  $\bar{z} \to x$ from the lower half plane.
Thus, since $k$ is an integer,
\[ g_+(x) = (-1)^{k+1} x^j  |x-w|^{\alpha} \int_x^w \frac{|s-w|^{\alpha + k}}{(1+xs)^{\beta + 1}} \, ds.
\]
With the same arguments, we find exactly the same expression for $g_-(x)$.
Thus $g$ has continuous extension
to $\mathbb C \setminus (-\infty, -1/w]$. The extension is in fact also
$C^{\infty}$ and from \eqref{fdef}
\begin{equation} \label{fzbar} 
\frac{\partial g}{\partial \bar{z}} = z^j (\bar{z}-w)^k \frac{|z-w|^{2\alpha}}{(1+|z|^2)^{\beta + 1}}.	 
\end{equation}

Let $R > 0$ be large and $\delta > 0$ small. We let $D_{R,\delta}$ be the part of the disk of radius $R$ around the origin that is at distance  $> \delta$ from $(-\infty, -1/w]$ (see Figure \ref{contorthfig}).
By Stokes theorem applied to \eqref{fdef} we then have by \eqref{fzbar}
\begin{equation}
\label{PnStokes}
\begin{aligned}
\int_{D_{R,\delta}}  z^j (\overline{z}-w)^k \frac{|z-w|^{2\alpha}}{(1+|z|^2)^{\beta+1}} \, dA(z)
&=\int_{D_{R,\delta}} \frac{\partial g}{\partial \bar{z}} \, dA(z)=\frac{1}{2i}\oint_{\partial D_{R,\delta}}g(z) \, dz
\\
& =\frac{1}{2i}\oint_{\partial D_{R,\delta}} z^j (z-w)^{\alpha} \int_w^{\overline{z}} \frac{(s-w)^{\alpha + k}}{(1+zs)^{\beta+1}} \, ds \, dz.
\end{aligned}\end{equation}
We write  the $s$-integral in the right hand side of \eqref{PnStokes} as
\begin{align} \label{split}
\int_w^{\overline{z}}
\frac{(s-w)^{\alpha + k}}{(1+zs)^{\beta+1}} \, ds 
= 
\int_w^{\overline{z} \times \infty}
\frac{(s-w)^{\alpha + k}}{(1+zs)^{\beta+1}} \, ds
- 
\int_{\overline{z}}^{\overline{z} \times \infty}
\frac{(s-w)^{\alpha + k}}{(1+zs)^{\beta+1}} \, ds . 
\end{align}
Both integrals converge for $z \in \partial D_{R,\delta}$ because
of the conditions \eqref{parcond}.

We split the integral of \eqref{PnStokes} accordingly.
We study the two terms
\begin{align} \label{term1} \frac{1}{2i}
\oint_{\partial D_{R,\delta}}
z^j (z-w)^{\alpha}
\int_w^{\overline{z} \times \infty}
\frac{(s-w)^{\alpha + k}}{(1+zs)^{\beta+1}} \, ds \, dz, 
\\
\label{term2} \frac{1}{2 i}
\oint_{\partial D_{R,\delta}}
z^j (z-w)^{\alpha}
\int_{\bar{z}}^{\bar{z} \times \infty}
\frac{(s-w)^{\alpha + k}}{(1+zs)^{\beta+1}} \, ds \, dz,
\end{align}
separately.
Let us first deal with the second term \eqref{term2} and  show that it vanishes in the limit $\delta \to 0+$
and $R \to +\infty$.

To deal with the limit $\delta \to 0+$, we compute the limiting values
of the integrand as $z \to x \in (-\infty,0)$. 
Let $x < 0$ and put  
\begin{equation} \label{hdef} h(z) =z^j 
	(z-w)^{\alpha}
	\int_{\overline{z}}^{\overline{z}\times \infty}
	\frac{(s-w)^{\alpha + k}}{(1+zs)^{\beta+1}} \, ds. \end{equation}
If we let $z \to x$
	with $\Im z > 0$, then
$ (z-w)^{\alpha} \to |x-w|^{\alpha} e^{\pi i\alpha}$,
and, since $\bar{z}$ tends to $x$ from below, the integral in \eqref{hdef} tends to
\[ 
\int_x^{-\infty} \frac{(s-w)^{\alpha+k}_-}{(1+xs)^{\beta + 1}} \, ds
=  -\int_{-\infty}^x
		\frac{|s-w|^{\alpha+k} e^{-\pi i (\alpha + k)}}{(1+xs)^{\beta+1}} \, ds.
\] 
Observe that $1+xs > 1$
for $-\infty < s < x < 0$
and we do not encounter any branching in the denominator.
The result is that
\[ h_+(x) = (-1)^{k+1}x^j |x-w|^{\alpha} \int_{-\infty}^x 	\frac{|s-w|^{\alpha+k}}{(1+xs)^{\beta+1}} \,  ds.  \]
We find the same result for $h_-(x)$,
and $h$ has continuous extension across $(-\infty,0)$. 

Thus, in the limit $\delta \to 0+$ the second term \eqref{term2} reduces to an integral over the circle of radius $R$ 
\begin{equation} \label{term2reduced} \frac{1}{2 i} \oint_{|z| = R}
	z^j (z-w)^{\alpha}
	\int_{\overline{z}}^{\overline{z}\times \infty}
	\frac{(s-w)^{\alpha + k}}{(1+zs)^{\beta+1}} \, ds \, dz
\end{equation}
whose absolute value is bounded by
\begin{equation} \label{term2est1}
	\pi R \sup_{|z| = R}
	\bigg|	z^j (z-w)^{\alpha}
	\int_{\overline{z}}^{\overline{z}\times \infty}
	\frac{(s-w)^{\alpha + k}}{(1+zs)^{\beta+1}} \, ds \bigg|  \leq C R^{N+\alpha + 1}
	\sup_{|z| = R}
	\bigg|
	\int_{\overline{z}}^{\overline{z}\times \infty}
	\frac{(s-w)^{\alpha + k}}{(1+zs)^{\beta+1}} \, ds \bigg|,
	\end{equation}
for some constant $C > 0$.
We have, with the substitution $s =\bar{z}/u$,
\[ \int_{\overline{z}}^{\overline{z}\times \infty}
\frac{(s-w)^{\alpha + k}}{(1+zs)^{\beta+1}} \, ds
= \bar{z} \int_0^1 
	 \frac{(\bar{z}-wu)^{\alpha + k}}{(u + |z|^2)^{\beta + 1}}
	u^{\beta-1-\alpha - k} \, du. \]
Since $\beta - 1 -\alpha - k \geq 0$
the integrand remains bounded for $u$
in the interval $[0,1]$, and
we readily obtain
\begin{equation} \label{term2est2} \sup_{|z| = R} \bigg| \int_{\overline{z}}^{\overline{z}\times \infty}
\frac{(s-w)^{\alpha + k}}{(1+zs)^{\beta+1}} \, ds \bigg|
	= \mathcal{O}(R^{-2\beta + \alpha + k - 1}) 
\end{equation}
as $R \to \infty$.
Combining \eqref{term2est1} and \eqref{term2est2}  we find that
\eqref{term2reduced} is $\mathcal{O}(R^{-2\beta +2 \alpha + N +k})$ as $R \to \infty$. Due to \eqref{parcond} and $k \leq N-1$, we have $-2\beta +2 \alpha + N +k \leq -1$, and we conclude that \eqref{term2reduced} vanishes in the large $R$ limit. The outcome  is therefore
\begin{equation} \label{term2limit}
	\lim_{R \to \infty} \lim_{\delta \to 0+}
 \frac{1}{2 i}
\oint_{\partial D_{R,\delta}}
z^j (z-w)^{\alpha}
\int_{\bar{z}}^{\bar{z} \times \infty}
\frac{(s-w)^{\alpha + k}}{(1+zs)^{\beta+1}} \, ds \, dz = 0.
\end{equation}

The first term \eqref{term1} has an
explicit evaluation for fixed $R$ and $\delta$. Consider ths $s$ integral in \eqref{term1} and change variables $s\mapsto s+w$
\begin{multline*}
\int_w^{\overline{z} \times \infty}
\frac{(s-w)^{\alpha + k}}{(1+zs)^{\beta+1}} \, ds 
=
\int_0^{\overline{z} \times \infty}
\frac{s^{\alpha + k}}{(1+z(s+w))^{\beta+1}} \, ds \\
= \frac{1}{(1+wz)^{\beta +1}}
\int_0^{\overline{z} \times \infty}
 s^{\alpha + k} 
 \Big(1+ \frac{zs}{1+wz}\Big)^{-\beta-1} \, ds.
\end{multline*}
Then with $s= x \bar{z}$ this is
\begin{align*}
 \frac{\bar{z}^{\alpha + k + 1}}{(1+wz)^{\beta +1}}
	\int_0^{\overline{z} \times \infty}
	x^{\alpha + k} 
	\Big(1+ \frac{|z|^2}{1+wz} x\Big)^{-\beta-1}\, dx. 
\end{align*}
Another change of variables 
$x = \frac{1+wz}{|z|^2} t$ leads to
\begin{equation}
	\frac{\bar{z}^{\alpha + k + 1}}{(1+wz)^{\beta +1}}
	\Big(\frac{1+wz}{|z|^2}\Big)^{\alpha + k + 1}
	\int_0^{\infty} \frac{t^{\alpha+k}}{(1+t)^{\beta+1}} \, dt= \frac{G_{\alpha+k,\beta}}{z^{\alpha + k+1}(1+wz)^{\beta - \alpha -k} },
\end{equation}
with constant 
$$
G_{\alpha+k,\beta} = \int_0^{\infty} \frac{t^{\alpha+k}}{(1+t)^{\beta+1}} \, dt= B(\alpha+k+1,\beta-\alpha-k) .
$$
The constant can be evaluated using the Euler's beta integral (cf. \cite[3.194,(3)]{GR1}). 

The first term \eqref{term1} is thus equal to
\[ \frac{G_{\alpha+k,\beta}}{2 i} \oint_{\partial D_{R,\delta}}
z^j \frac{(z-w)^{\alpha}}{z^{\alpha + k+1}(1+wz)^{\beta - \alpha -k} } \, dz. \]
Observe that $\frac{(z-w)^{\alpha}}{z^{\alpha+k+1}}$
has a branch cut on $[0,w]$ only, and
the integrand is analytic in
$\mathbb C \setminus ((-\infty,-1/w] \cup [0,w])$. By analyticity
$\partial D_{R,\delta}$, can then be deformed
to any contour $\Gamma$ going around
$[0,w]$ once in the positive direction,
and not intersecting $(-\infty, -1/w]$. We can take $\Gamma$ independent of $R$ and $\delta$, and we find
\begin{multline} \label{term1eval}
	\frac{1}{2\pi i}
	\oint_{\partial D_{R,\delta}}
	z^j (z-w)^{\alpha}
	\int_w^{\overline{z} \times \infty}
	\frac{(s-w)^{\alpha + k}}{(1+zs)^{\beta+1}} \, ds \, dz
	\\ =
	\frac{G_{\alpha+k,\beta}}{2 i}
\oint_{\Gamma}
z^j  \frac{(z-w)^{\alpha}}
{z^{\alpha + k+1}(1+wz)^{\beta - \alpha -k} } \, dz. 
\end{multline}

Then combining everything, we have
\begin{align*}  & \quad \int_{\mathbb C} 
	z^j (\overline{z}-w)^k
	\frac{|z-w|^{2\alpha}}{(1+|z|^2)^{\beta+1}} \, dA(z) \\
	&  = 
	\lim_{R \to +\infty} \lim_{\delta \to 0+}   \int_{D_{R,\delta}} 
	z^j (\overline{z}-w)^k
	\frac{|z-w|^{2\alpha}}{(1+|z|^2)^{\beta+1}} \, dA(z) \\
	& =
	\lim_{R \to +\infty} \lim_{\delta \to 0+}   G_{\alpha+k,\beta} \oint_{\partial D_{R,\delta}} 
	z^j (z-w)^\alpha
	\int_w^{\bar{z}} 
	\frac{(s-w)^{\alpha+k}}{(1+zs)^{\beta+1}} \, ds \, dz.\\&
    =\frac{G_{\alpha+k,\beta}}{2 i}
\oint_{\Gamma}
z^j  \frac{(z-w)^{\alpha}}
{z^{\alpha + k+1}(1+wz)^{\beta - \alpha -k} } \, dz.
\end{align*}
Replacing $\alpha, \beta$ in terms of $N,Q_1, Q_0$ as given in \eqref{short} and rewriting $G_{\alpha+k,\beta}=G_{k,N}$ and  gives \eqref{lemma1}.
\end{proof}

With  Lemma \ref{contorth0} we are now ready to show that the planar orthogonal polynomials also satisfy non-Hermitian contour orthogonality.

\begin{proof}[Proof of Proposition \ref{contorth}]
If \eqref{eq9} holds true then from Lemma \ref{contorth0} along with the fact $G_{k,N}\neq 0$, we have for a contour $\Gamma$ as in Lemma  \ref{contorth0},
 \begin{equation}\label{north}
    \oint_{\Gamma }P_{n,N}(z)\frac{(z-w)^{NQ_1}}{(1+wz)^{N+NQ_{0}-k} }\frac{ dz}{z^{k+1+NQ_{1}}}=0,
\end{equation}
for $k=0,\dots n-1.$ 
Then \eqref{north} can be rewritten as 
 \begin{equation}\label{north2}
    \oint_{\Gamma }P_{n,N}(z)(1+wz)^{k}z^{n-k-1}\frac{(z-w)^{NQ_1}}{(1+wz)^{N+NQ_{0}}}\frac{dz}{z^{n+NQ_{1}}}\, =0,
\end{equation}
for $k=0,\dots n-1.$ 

The family of polynomials $k\mapsto(1+wz)^{k}z^{n-k-1}$ $k=0,1,\dots,n-1$, are a basis of the vector space of polynomials of degree $\leq n-1$. Thus, by taking appropriate linear combination we obtain \eqref{nonh}. 
\end{proof}

\section{Riemann-Hilbert analysis}\label{RHsection}

By Proposition \ref{contorth} the planar orthogonal polynomials $P_{n,N}$ satisfy non-Hermitian contour orthogonality, and consequently by the renowned work of \cite{FIK92} $P_{n,N}$ is a solution to the Riemann-Hilbert (RH) problem described below.

The RH problem is for the orthogonality contour $\Gamma$ (now fixed as given in Theorem \ref{thmmoth}). 
Namely, $\Gamma = \Gamma_0 \cup \Gamma_1 \cup \Gamma_2$, where $\Gamma_0$ is defined in Definition~\ref{defmoth}, and $\Gamma_1$ and $\Gamma_2$ are defined in Lemma~\ref{orthtraj} with coubterclockwise orientation.

\begin{rhproblem} \label{rhproblemY} $Y$ is the solution of the following RH problem.

\begin{itemize}
\item $Y:\mathbb C\setminus\Gamma\rightarrow \mathbb C^{2\times 2}$ where $\Gamma$ is analytic.
    \item On $\Gamma $ we have 
    $$
    Y_+(z)=Y_-(z)\begin{pmatrix}
        1&w_{n,N}(z)\\
        0&1
    \end{pmatrix},
    $$
    where $Y_\pm$ denotes the boundary values of $Y$ on $\Gamma$. Here, $w_{n,N}$ is given by \eqref{wC}.
    \item As $z\rightarrow \infty$ we have \begin{equation}
        Y(z)=\left(I_2+\mathcal{O}(z^{-1})\right)\begin{pmatrix}
            z^{n}&0\\
            0& z^{-n} 
        \end{pmatrix}.
    \end{equation}
\end{itemize}

\end{rhproblem}
Then we have $P_{n,N}(z)=Y_{1,1}(z)$.

We also denote the squared norm associated with the non-Hermitian contour orthogonality by \begin{equation}\label{RHnorm} 
    \oint_{\Gamma}  P_{n,N}^2(z)  \frac{(z-w)^{NQ_1}}{(z+1/w)^{N+NQ_0}}\frac{dz}{z^{n+NQ_1}}=  \widehat{h}_{n,N}. 
\end{equation} 
Then it is well known that (see for example \cite{D00})
\begin{equation}\label{rhnorm}
    \widehat{h}_{n,N}=-2\pi i\lim_{z\rightarrow \infty}z^{n+1}Y_{1,2}(z).
\end{equation}

\subsection{Some preliminaries}
At this point, one should observe from \eqref{defV} and \eqref{wC} that
\begin{equation} \label{w nN mathcal V}
        w_{n,N}(z)=\frac{1}{z^{r_0}}e^{-N \mathcal{V}(z)},
\end{equation}
where we also recall from the statement of Theorem~\ref{thm1} that $r_0=n-N$.

\begin{definition}\label{defphi}
We define 
  \begin{equation}
    \varphi(z) : =\int_{z_1}^{z}\sqrt{R_0(s)}\, ds,
\end{equation}
  where the contour of integration lies in $\mathbb C \setminus \left(\Gamma_0\cup\Gamma_1\cup[c_0,\infty)\right)$. Then we also have from \eqref{defU0}, \eqref{defR0} and \eqref{def of widetilde U U0} that $\widetilde{\mathcal{U}}_{0}(z)=\Re \varphi(z)$.
\end{definition}

\begin{definition}
We define the $g$ function as 
    \begin{equation}\label{defg}
        g(z):= \int \log(z-s) \, d\mu_0(s), 
    \end{equation}
    where the branch cut of the logarithm is taken in $\Gamma_0\cup\Gamma_1\cup[c_0,\infty)$.
\end{definition}

Notice that from \eqref{defg} and \eqref{tmass2}, we have 
\begin{equation}\label{gsymp}
    g(z)=\log z +\mathcal{O}(z^{-1}) \quad z\rightarrow \infty.
\end{equation}
On the other hand, by \eqref{def of U mu and I mu} and \eqref{def of Cauchy trans}, we have 
$$\Re g(z)= - U^{ \mu_0 }(z), \qquad g'(z)=C^{ \mu_0 }(z).$$ 
Then it follows from \eqref{c2} that 
$$
2g'(z) = \mathcal{V}'(z) - \sqrt{R_0(z)}. 
$$
Then by \eqref{eqcon}, we obtain
\begin{equation}\label{gphi}
    2 g(z)=\mathcal{V}(z)-\varphi(z)+\ell_0.
\end{equation}
Recall that $\mathcal{U}_0$ is defined by \eqref{defU0}. 
Consequently from item (2) of Theorem \ref{thmmoth}
\begin{align}
\begin{split}\label{jcon}
w_{n,N}(z)e^{N\left(g_+(z)+g_{-}(z)-\ell_0\right)}
&= z^{-r_0}  e^{-N/2\left(\varphi_{+}(z)+\varphi_{-}(z)\right)}
= \begin{cases}
    z^{-r_0}, & z\in \Gamma_0,\\
   z^{-r_0} e^{-N \mathcal{U}_0(z)}, & z\in \Gamma \setminus \Gamma_0.
\end{cases}
\end{split}
\end{align}

\subsection{\texorpdfstring{Transformation with the $g$ function}{}}

Our first transformation involves the $g$ function defined in \eqref{defg}. This normalizes the behaviour at infinity due to \eqref{gsymp} and makes the $(1,2)$ entry of the jump on the contour $\Gamma_0$ to be $z^{-r_0}$ due to \eqref{jcon} while exponentially small on the rest of $\Gamma$. 
 
\begin{definition}
    We define 
\begin{equation}
    X(z):=e^{-\frac{N}{2} \ell_0\sigma_3}Y(z)e^{-N\left(g(z)-\frac{\ell_0}{2}\right)\sigma_3}.
\end{equation}
Here $\sigma_3=\begin{pmatrix}
    1&0\\
    0&-1
\end{pmatrix}$, is the third Pauli matrix. 
\end{definition}
Then it follows from the RH problem~\ref{rhproblemY}, \eqref{jcon}, and \eqref{gsymp} that we have the following RH problem for $X$.

\begin{rhproblem}\label{gfunction}
 $X$ is the solution to the following RH problem.  
 \begin{itemize}
\item $X:\mathbb C\setminus\Gamma\rightarrow \mathbb C^{2\times 2}$ is analytic.

    \item $X_+=X_-J_X$ on $\Gamma$ with 
    \begin{equation}
    J_X :=    \begin{cases}
    \begin{pmatrix}
        e^{-N\left(g_-(z)-g_{+}(z)\right)}& z^{-r_0}\\
        0&e^{N\left(g_-(z)-g_{+}(z)\right)}\end{pmatrix},  \qquad & \mathrm{on} \hspace{0.2cm} \Gamma_0,
    \smallskip 
    \\
    \begin{pmatrix}
        1&z^{-r_0}e^{-N \mathcal{U}_{0}(z)}\\
        0&1
    \end{pmatrix}, \qquad & \mathrm{on} \hspace{0.2cm}\Gamma \setminus\Gamma_0. 
    \end{cases}
    \end{equation}
    \item As $z\rightarrow \infty$ we have \begin{equation}
       X(z)=\left(I_2+\mathcal{O}(z^{-1})\right)\begin{pmatrix}
           z^{r_0}&0\\
           0& z^{-r_0}
       \end{pmatrix}.
        \end{equation}
\end{itemize}

\end{rhproblem}

\subsection{Opening the lenses}

The RH Problem \ref{gfunction} has oscillatory diagonal entries on the cut $\Gamma_0$. We open two lenses $L_+, L_-$ on either side of $\Gamma_0$. The condition on the lens is that we must have $\Re \varphi(z)<0$ in the interior and the boundary  of the lenses. This can be done in a neighborhood of $\Gamma_0$, as the measure $\mu_0$ has a strictly positive density on the interior of $\Gamma_0$. 
\begin{definition}
    We define
\begin{equation}
    T(z):= X(z) \begin{cases}
        I_2, & z\in \mathbb C \setminus L_\pm,
        \smallskip 
        \\
        \begin{pmatrix}
            1&0\\
           \mp z^{r_0}e^{N\varphi(z)}&1
        \end{pmatrix}, & z\in L_{\pm}.\\     
    \end{cases}
\end{equation}
\end{definition}

\begin{rhproblem}\label{RHPT} $T$ satisfies the following RH problem. 
     \begin{itemize}
\item $T:\mathbb C\setminus\Sigma_T\rightarrow \mathbb C^{2\times 2}$ is analytic where $\Sigma_T=\Gamma\cup \partial L_\pm$.
    \item $T_+=T_-J_T$ on $\Sigma_T$ with 
    \begin{equation} \label{def of JT}
    J_T:= \begin{cases}
    \begin{pmatrix}
       0&z^{-r_0}\\
        -z^{r_0}&0
    \end{pmatrix},  \qquad & \mathrm{on} \hspace{0.2cm} \Gamma_0,
    \smallskip 
    \\
    \begin{pmatrix}
        1&z^{-r_0}e^{-N \mathcal{U}_{0}(z)}\\
        0&1
    \end{pmatrix}, \qquad &  \mathrm{on} \hspace{0.2cm}\Gamma \setminus\Gamma_0, 
    \smallskip 
    \\
    \begin{pmatrix}
            1&0\\
           z^{r_0} e^{N\varphi(z)}&1 
        \end{pmatrix}, \qquad  &  \mathrm{on} \hspace{0.2cm}  \partial 
        L_\pm.
    \end{cases}
    \end{equation}  
    \item As $z\rightarrow \infty$, we have \begin{equation}
       T(z)=\left(I_2+\mathcal{O}(z^{-1})\right)\begin{pmatrix}
           z^{r_0}&0\\
           0& z^{-r_0}
       \end{pmatrix}.
    \end{equation}
\end{itemize}
\end{rhproblem}

Since $r_0$ is fixed, observe that away from $z_1$ and $z_2$, the jumps on $\Gamma \setminus \Gamma_0$ and $\partial L_\pm$ converge to the identity matrix exponentially fast as $N \to \infty$.
On the other hand, on $\Gamma_0$ it has simple jumps. This motivates the construction of the global matrix, which approximates $T(z)$ away from the branch points.

\subsection{Global parametrix}

\begin{rhproblem}\label{RHM} We seek a matrix $M$ that satisfies the following conditions.

    \begin{itemize}
\item $M:\mathbb C\setminus\Gamma\rightarrow \mathbb C^{2\times 2}$ is analytic. 
  \item On $\Gamma_0 $ we have $M_+(z)=M_-(z)J_M$, with 
  \begin{equation}\label{JM}
      J_M:=\begin{pmatrix}
       0&z^{-r_0}\\
        -z^{r_0}&0
    \end{pmatrix}.
  \end{equation}
    \item As $z\rightarrow x^*$, with $x^{*}  \in  \{z_1,z_2\}$, we have 
    $$M(z)=\mathcal{O}((z-x^*)^{1/4}).$$
     \item As $z\rightarrow \infty$ we have \begin{equation}\label{Mexp}
       M(z)=\left(I_2+\mathcal{O}(z^{-1})\right)\begin{pmatrix}
           z^{r_0}&0\\
           0& z^{-r_0}
       \end{pmatrix}.
    \end{equation}
\end{itemize}
\end{rhproblem}

\begin{lemma}\label{Dlemma}
Define
\begin{equation}\label{defD}
         D(z):=\sqrt{\frac{\rho}{a}}\frac{1}{F_1(z)} ,
\end{equation}
 where $F_1$, one of the inverses of the map $f$, is given by \eqref{defF1}.      
Then it satisfies the following.
    \begin{itemize}
        \item[(1)] $D(z)$ is analytic and nonzero in $\mathbb C \setminus\Gamma_0$;
        \item [(2)] $D_{+}(z)D_{-}(z)=z $ if  $z\in \Gamma_0$;
        \item [(3)] $D(z) 
        =D_{\infty}z+\mathcal{O}(1)$ as $ z\rightarrow\infty $, where $D_{\infty}:=\frac{1}{\sqrt{\rho a}}. $
    \end{itemize}
    
    \end{lemma}

\begin{proof}
    Item (1) is immediate from definition  \eqref{defD}. Item (3) follows from \eqref{invasy}. 
We prove item (2). By direct computation using \eqref{confm} and \eqref{deck rational}, we have 
\begin{equation}
    f(u)=\frac{\rho}{a}\frac{1}{  \mathrm{deck}(u)u}.
\end{equation}
Notice that by \eqref{F12 deck}, we have  
$$F_1(z_{\pm})=F_2(z_\mp)=\mathrm{deck}(F_1(z_\mp)), \quad z\in \Gamma_0.$$
Thus we obtain 
  \begin{equation}
      D_{+}(z)D_{-}(z)=\frac{\rho}{a}\frac{1}{F_1(z)}\frac{1}{\mathrm{deck}(F_1(z))}=z, \quad z\in \Gamma_0.
  \end{equation}
\end{proof}

Next, we consider  
\begin{equation}\label{newrhp}
    M(z)=D_{\infty}^{-r_0\sigma_3}N(z)D(z)^{r_0\sigma_3},
\end{equation}
and the Riemann-Hilbert problem associated with $N(z)$.

\begin{rhproblem}\label{RHN} Following RH problem~\ref{RHM} and \eqref{newrhp}, we seek a matrix $N$ that satisfies the following conditions.
    \begin{itemize}
\item $N:\mathbb C\setminus\Gamma\rightarrow \mathbb C^{2\times 2}$ is analytic.
  \item On $\Gamma_0 $ we have $N_+(z)=N_-(z)J_N$, with 
  \begin{equation}\label{JN}
      J_N:=\begin{pmatrix}
       0&1\\
        -1&0
    \end{pmatrix}.
  \end{equation}
    \item As $z\rightarrow x^*$, with $x^{*} \in \{z_1,z_2\}$, we have 
    $$N(z)=\mathcal{O}((z-x^*)^{1/4}).$$
     \item As $z\rightarrow \infty$ we have \begin{equation}\label{Nasy1}
       N(z)=I_2+\mathcal{O}(z^{-1}).
    \end{equation}
\end{itemize}
\end{rhproblem}
Observe that in the case $r_0=0$ the RH problem \ref{RHN} is same as the RH problem \ref{RHM}, that is $N(z)=M(z)$.
We construct the solution of RH Problem \ref{RHN} by means of the conformal map $F_1$ defined in \eqref{defF1}.
Denote 
$$
\widetilde{\Gamma}_{\pm}:= F_1(\Gamma_{0\pm})=F_2(\Gamma_{0\mp}).
$$
We are looking for a solution of the form 
\begin{equation} \label{def of N ansatz} 
    N(z)=\begin{pmatrix}
        N_{1}(F_1(z))&N_{1}(F_2(z))\\
        N_{2}(F_1(z))&N_{1}(F_2(z))
    \end{pmatrix}, 
\end{equation}
where $N_1$ and $N_2$ are two scalar-valued functions on $\mathbb C\setminus \widetilde{\Gamma}_{\pm}$.

In order to satisfy the jump condition \eqref{JN} we must have 
\begin{equation}
N_{k_+}(u) = \begin{cases}
-N_{k_-}(u), & u \in \widetilde{\Gamma}_+, 
\\
N_{k_-}(u), &  u \in \widetilde{\Gamma}_-. 
\end{cases}
\end{equation}
Hence $N_k$ is analytic across $\widetilde{\Gamma}_-$.
We recall from \eqref{invasy} that as $z\rightarrow \infty$
\begin{equation}
    F_{1}(z)=\frac{\rho}{z}+\mathcal{O}(z^{-2}),\qquad   F_{2}(z)=\frac{1}{a}+\mathcal{O}(z^{-1}),
\end{equation}
and a direct computation shows in \eqref{confm}
\begin{equation}
    \frac{1}{\sqrt{f'(z)}}=(z-1/a)\frac{i}{\sqrt{a(b-a)\rho}}+\mathcal{O}((z-1/a)^2),\quad z\rightarrow 1/a,
\end{equation}
\begin{equation}
    \frac{1}{\sqrt{f'(z)}}=iz\frac{1}{\sqrt{\rho}}+\mathcal{O}(z),\quad z\rightarrow 0.
\end{equation}
With this in mind, we set 
\begin{equation}\label{gpar2}
N_{1}(u):=a_1\frac{1}{\sqrt{f'(u)}}\frac{1}{u}, \qquad 
       N_{2}(u):=a_2\frac{1}{\sqrt{f'(u)}}\frac{1}{u-1/a}, 
\end{equation}
where 
\begin{equation}\label{gpar3}
a_1=i\sqrt{\rho}, \qquad a_2=-i\sqrt{\frac{a}{b-a }\frac{1}{\rho }},
\end{equation}
and  the branch of the square root of $\sqrt{f'(u)}$ is taken on $\widetilde{\Gamma}_+$.

Since $f(F_{j}(z))=z$ $(j=1,2)$ we have the chain rule 
$$ \frac{1}{\sqrt{f'(F_{j}(z))}}=\sqrt{F_{j}'(z)}.$$ 
Then by \eqref{def of N ansatz} and \eqref{gpar2}, we have
\begin{equation}\label{gpar}
    N(z)=\begin{pmatrix}
        a_1\frac{\sqrt{F_{1}'(z)}}{F_1(z)}& a_1\frac{\sqrt{F_{2}'(z)}}{F_2(z)}\\
        a_2\frac{\sqrt{F_{1}'(z)}}{F_1(z)-1/a}&a_2\frac{\sqrt{F_{2}'(z)}}{F_2(z)-1/a}
    \end{pmatrix}.
\end{equation}
As a consequence, by \eqref{newrhp}, we obtain 
\begin{equation}\label{Mform}
    M(z)=D_{\infty}^{-r_0\sigma_3}\begin{pmatrix}
        a_1\frac{\sqrt{F_{1}'(z)}}{F_1(z)}& a_1\frac{\sqrt{F_{2}'(z)}}{F_2(z)}\\
        a_2\frac{\sqrt{F_{1}'(z)}}{F_1(z)-1/a}&a_2\frac{\sqrt{F_{2}'(z)}}{F_2(z)-1/a}
    \end{pmatrix}
D(z)^{r_0\sigma_3}.
\end{equation}

\subsection{Local parametrices}
Local parametrices are defined in a neighbourhood of the branch points $z_j$ and are constructed in the discs $\mathbb{D}(z_j, \delta)$ for $j \in \{1,2\}$.
We take 
\begin{equation}
    D=\mathbb D (z_{1},\delta)\cup \mathbb D (z_{2},\delta).
\end{equation}
For $\delta>0$ small, the local parametrices are defined in $D$, with jump matrices that agree with $J_T$ given in \eqref{def of JT}. In addition, it agrees with $M$ on the boundary of the discs up to an error $\mathcal{O}(n^{-1})$ as $n\rightarrow\infty$. Specifically, we introduce $P(z)$ according to the following RH problem.
\begin{rhproblem}
$P$ satisfies the following RH problem. 
    \begin{itemize} 
    \item $P(z)$ analytic in $D\setminus \Sigma_P$ where $\Sigma_P:=\Sigma_T\cap D$. (Recall $\Sigma_T$ in RH problem \ref{RHPT}.) 
      \item $P_{+}=P_- J_P$ on $\Sigma_P$ with 
    \begin{equation}
    J_P:= 
    \begin{cases}
    \begin{pmatrix}
       0&z^{-r_0}\\
        -z^{r_0}&0
    \end{pmatrix}, & \quad \mathrm{on}\hspace{0.2cm} \Gamma_0\cap D,
    \smallskip 
    \\ 
    \begin{pmatrix}
        1&z^{-r_0}e^{-N \mathcal{U}_{0}(z)}\\
        0&1 
    \end{pmatrix}, & \quad \mathrm{on}\hspace{0.2cm} \Gamma\setminus \Gamma_0 \cap D,
    \smallskip 
    \\ 
    \begin{pmatrix}
            1&0\\
           z^{r_0}e^{N\varphi(z)}&1 
        \end{pmatrix}, & \quad \mathrm{on}\hspace{0.2cm} \partial L_\pm \cap D.
    \end{cases}
    \end{equation} 
      \item $P(z)$ matches with the outer parametrix $M$ that is 
\begin{equation}\label{match}
          P(z)=M(z)\left(1+\mathcal{O}(n^{-1})\right) \quad n\rightarrow \infty,
      \end{equation}
   uniformly as $z\in \partial D$. \end{itemize}
\end{rhproblem}

The functions $\varphi(z)=\mathrm{const}\times (z-z_{j})^{3/2}\left(1+\mathcal{O}(z-z_{j})^{1/2}\right)$ and the local parametrices of $z^{(r_0/2)\sigma_3}P(z) z^{(-r_0/2)\sigma_3}$ can be constructed out of Airy functions; see for instance \cite{BBLM15, D00}. We omit the construction here, as the precise form is not necessary for our  purpose.

\subsection{Final transformation}

\begin{definition}
We make the final transformation 
    \begin{equation}
    S(z):=\begin{cases}
        T(z)M(z)^{-1}, &\quad z\in \mathbb C \setminus\left( D \cup \partial L_{\pm}\cup \Gamma\right)
        \smallskip 
        \\
        T(z)P(z)^{-1}, & \quad z\in D \setminus( \partial L_{\pm}\cup \Gamma).
    \end{cases}
\end{equation}
\end{definition}
\begin{figure}
    \centering
    \includegraphics[width=0.5\linewidth]{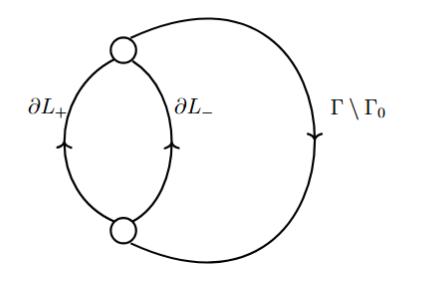}
    \caption{The jump contour $\Sigma_S$}
    \label{JSjump}
\end{figure}

The jump matrices of $T$ and $M$ coincide on $\Gamma_0$, while the jump matrices of $T$ and $P$ coincide on $D$. It follows that $S(z)$ has an analytic continuation to $\mathbb{C} \setminus \Sigma_S$, as shown in Figure \ref{JSjump}. The matching condition \eqref{match} ensures that $PM^{-1} = I_2 + \mathcal{O}(N^{-1})$ uniformly as $N \to \infty$ on $\partial D$.
Thus, as $N \to \infty$,   
\begin{equation}
    S_+(z)=S_{-}(z)\left(I+\mathcal{O}(N^{-1}) \right), 
\end{equation}
uniformly on $\partial D$. 
On the rest of the contour $\Sigma_S$, as $N \to \infty$, there exists $c>0$ such that 
\begin{equation}
    S_+(z)=S_{-}(z)\left(I+\mathcal{O}(e^{-cN}) \right), 
    \end{equation}
    uniformly on $\partial D$. 
Since $ S(z)\rightarrow I_2$ as $ z \to \infty$, standard arguments yield that as $N \to \infty,$ 
\begin{equation}\label{Sasy}
    S(z)=I+\mathcal{O} \Big(\frac{1}{N(1+|z|)}\Big), 
\end{equation}
uniformly in $z\in \mathbb C\setminus \Sigma_S$.

\subsubsection{Proof of Theorem \ref{thm1}}

\begin{proof}
We take $z$ in a compact subset $K$ of $\overline{\mathbb C}\setminus\Gamma_0$. We may further assume that $L_{\pm}$ and $\mathbb D( z_j,\delta)$ are such that $K$ is in their exterior.

Following the transformations $$Y\mapsto X \mapsto T \mapsto S$$ we obtain for $z\in K$, 
\begin{equation}\label{Yasy}
    Y(z)=e^{N\frac{\ell_0}{2}\sigma_3}S(z) M(z)e^{N(g(z)-\frac{\ell_0}{2})\sigma_3},
\end{equation}
where $S(z)$ satisfies \eqref{Sasy}. Now recalling $P_{n,N}=Y_{1,1}$ we obtain the asymptotic formula 
\begin{equation}\label{strongasy2}
\begin{aligned}
     P_{n,N}(z)=e^{Ng(z)}M_{1,1}(z)(1+\mathcal{O}(N^{-1})), \quad z\in \mathbb C \setminus \left(D\cup L_{\pm}\right),
    \end{aligned}
\end{equation}
uniformly for  $z\in K$. By computing the $(1,1)$-entry of the matrix $M$ given in \eqref{Mform}, we obtain \eqref{strongasy}.
\end{proof}

\subsubsection{Proof of Theorem \ref{thm2}}
\begin{proof}

The proof is immediate from the strong asymptotic formula \eqref{strongasy}.  By \eqref{gpar2}, we have $M_{1,1}(z)\neq 0$ for $z\in \mathbb C \setminus \Gamma_0$. Hence, the zeros of $P_{n,N}$ accumulate on $\Gamma_0$ as $N\rightarrow\infty$. 

From \eqref{strongasy} we further conclude that 

\begin{equation}
    \lim _{n\rightarrow\infty}\frac{1}{n}\log |P_{n,N}(z)|=\Re g(z)=\int \log|z-s| \, d\mu_0(s), 
\end{equation}
uniformly for compact subsets of $\mathbb C \setminus \Gamma_0$. It is well known \cite{ST97} that this implies $\mu_0$ is the weak limit of the normalized zero counting measure; see also \cite{KKL24}.
\end{proof}

\section{Proof of Theorem \ref{norma}}\label{orthsection} 

Recall $\ell_{\mathrm{2D}}$ in \eqref{frost1}.
We first relate the constants $\ell_{\mathrm{2D}}$ and $\ell_0$ appearing in Theorem \ref{thmmoth}. 
We need a preliminary lemma. Recall $\mathcal{U}$ and $\mathcal{U}_0$ from \eqref{defU} and \eqref{defU0} and their scalings \eqref{def of widetilde U U0}. 
\begin{lemma}
Let $s=\frac{b \rho^2}{a z}$, then as  $z\rightarrow 0$ we have
    \begin{equation}\label{pot3}
    \begin{aligned}
        2\widetilde{\mathcal{U}}_0(z)= 2(1+Q_1+Q_0)&\log\Big(1+\frac{b\rho^2}{a}\Big)+\widetilde{\mathcal{U}}(z)+\widetilde{\mathcal{U}}(s)\\&
      -(1+Q_0+Q_1) \log(1+|s|^2)+\mathcal{O}(z).
       \end{aligned}
    \end{equation}
\end{lemma}

\begin{proof}
    Recall we have from \eqref{defU0} $2\mathcal{U}_0(z)=2\Re \int_{z_1}^z\sqrt{R(s)}ds$. Denote $F_1(z)=v_1$. 
Then following the same steps of Proposition \ref{traj1} we have
  \begin{align} 
    \begin{split}
 2\int_{z_1}^{z}\sqrt{R(s)}\, ds= 2\int_{z_1}^{z}(S_{2}(s)-S_{1}(s)) \, ds
    =2\int_{v_1}^{\mathrm{deck}(v_1)}\frac{f(1/u)}{1+f(1/u)f(u)}\, d(f(u)). 
    \end{split}
    \end{align} 
    Also recall $z_0=f(1)$, then as before we decompose the last integral as $I_1+I_2$, where
\begin{equation}
    I_1=\int_{v_1}^{z_0}\frac{f(1/u)}{1+f(1/u)f(u)}\,d(f(u)),\qquad  I_2=\int_{z_0}^{\mathrm{deck}(v_1)}\frac{f(1/u)}{1+f(1/u)f(u)}\,d(f(u)).
\end{equation}
By definition, we have
\begin{equation}\label{I_1}
    I_1=-\int_{z_0}^{z}S_1(s)\,ds. 
\end{equation}

We turn our attention to $I_2$. We have 
    \begin{align}\label{I_2}
    \begin{split}
    &I_2=\int_{1}^{\mathrm{deck}(v_1)}d\log (1+f(1/u)f(u))+\frac{f(u)f'(1/u)}{u^2(1+f(1/u)f(u) ) } \, du
    \\
    &=\Big[\log (1+f(1/u)f(u))\Big]_{v_0}^{\mathrm{deck}(v_1)}+\int_{1/\mathrm{deck}(v_1)}^{1}\frac{f(1/x)f'(x)}{1+f(1/x)f(x)}\,dx.
    \end{split}
    \end{align}
The above steps are similar to the proof of Proposition \ref{traj1}.  We now deviate from it.   As $z\rightarrow 0$ we have $v_1\rightarrow 1/b$ and $\mathrm{deck}(v_1)\rightarrow\infty$. However we have from \eqref{confm} 
\begin{equation}\label{lim1}\lim_{u\rightarrow\infty}f(1/u)f(u)=\frac{\rho^2 b}{a}.\end{equation}

Finally another expansion near $0$ shows, as $ z\rightarrow 0$, 
\begin{equation}\label{lim2} 
f\Big(\frac{1}{\mathrm{deck}(F_1(z))}\Big) =\frac{\rho^2 b}{a z}+\mathcal{O}(1).
\end{equation}
Recall $s= \frac{\rho^2 b}{a z}$ then as $z\rightarrow 0$, then plugging \eqref{lim1} and \eqref{lim2} in \eqref{I_2}, we find that RHS of \eqref{I_2} can be written as 
\begin{equation}\label{I3}
    \log\Big(1+\frac{\rho^2 b}{a}\Big)-\log (1+|z_0|)^2-\int_{z_0}^{s}S_{1}(t)\, dt+\mathcal{O}(z).
\end{equation}
We also recall the definition of $\mathcal{U}$ from \eqref{defU}. Then
combining \eqref{I3} and \eqref{I_1} and taking the real part with the above definition we obtain 
\begin{equation}
   2 \mathcal{U}_0(z)= 2\log\Big(1+\frac{\rho^2 b}{a}\Big)-\log(1+|s|^2)+\mathcal{U}(s)+\mathcal{U}(z).
\end{equation}
Scaling by $(1+Q_0+Q_1)$ yields the result.
\end{proof}

\begin{proposition}\label{lemmrela}
 We have the following relation between $\ell_{\mathrm{2D}}$ and $\ell_0$: 
  \begin{equation}\label{relation}
    \begin{aligned}
     \ell_0   =\ell_{\mathrm{2D}} & +(1+Q_0)\log w + \Re g(0)
     \\
     &\quad -(1+Q_1)\log (1+Q_1)-Q_0\log Q_0 +(1+Q_0+Q_1)\log \left(1+Q_0+Q_1\right).
   \end{aligned}
   \end{equation}
\end{proposition}

\begin{proof}

By \eqref{def of widetilde U U0}, \eqref{defU0} and \eqref{c2}, the LHS of \eqref{pot3} can be written as 
\begin{equation*}
    \begin{aligned}
2\widetilde{\mathcal{U}}_0(z)=4U^{\mu_0}(z)+2(1+Q_1)\log |z|+2(1+Q_0)\log |z+1/w|- 2Q_1\log |z-w|+2\ell_0.
\end{aligned}
\end{equation*}
Then by \eqref{defg}, we have  
\begin{equation}\label{cmp1}
2\widetilde{\mathcal{U}}_0(z)= 2(1+Q_1)\log |z|-2(1+Q_0)\log w- 2Q_1\log w-4 \Re g(0)+2 \ell_0+\mathcal{O}(z),
\end{equation}
as $z \to 0$.

Next, we compute the asymptotic behaviour of the RHS of \eqref{pot3} as $z \to 0$. 
By \eqref{defU} we can write
$$\widetilde{\mathcal{U}}(z)=2U^{\mu_0}(z)-2Q_1\log |z-w|+(1+Q_0+Q_1)\log(1+|z|^2)+\ell_{2D}.$$
Therefore, we have  
\begin{equation}\label{RHS1}
\widetilde{\mathcal{U}}(z)=-2Q_1\log w-2\Re g(0)+\ell_{2D}+\mathcal{O}(z),
\end{equation}
as $ z \to 0$. On the other hand, we have  
\begin{equation}
\begin{aligned}\label{RHS2}
\widetilde{\mathcal{U}}(s)
=2 Q_0\log |s|+\ell_{\mathrm{2D}}+\mathcal{O}(1/s),
\end{aligned}
\end{equation}
as $s \to \infty$. 
Hence combining \eqref{RHS1} and \eqref{RHS2} we find that the RHS of \eqref{pot3} has the expansion
    \begin{equation}\label{comp2}
    \begin{aligned}
        -2 Q_1 \log w-2 \Re g(0)&+2\ell_{\mathrm{2D}}+2(1+Q_1)\log |z|-2(1+Q_1)\log \frac{b\rho^2}{a}\\&+2(1+Q_0+Q_1)\log \Big(1+\frac{b\rho^2}{a}\Big)+\mathcal{O}(z),
   \end{aligned}
         \end{equation}
as $z \to 0$.  
Now substituting \eqref{a b ratio Q0Q1} and
equating \eqref{cmp1} with \eqref{comp2} gives \eqref{relation}.
\end{proof}

Recall that $h_{n,N}$ is the squared norm in the planar orthogonality \eqref{F11}. For convenience in proving the next lemma, we also denote
\begin{equation} \label{h0}
    \oint_{\Gamma}  P_{n,N}^2(z)  \frac{(z-w)^{NQ_1}}{(1+zw)^{N+NQ_0}}\frac{dz}{z^{n+NQ_1}}=  \widetilde{h}_{n,N}. 
\end{equation} 

\begin{lemma}\label{rela1}
We have 
\begin{equation}\label{h1}
    \widetilde{h}_{n,N}=-\frac{2i}{G_{n,N}}P_{n+1,N}(0)\, h_{n,N}, 
\end{equation}
where $G_{n,N}$ is defined by \eqref{defbeta}. 
\end{lemma}

\begin{proof}

We denote by $\mathcal{P}$ be the vector space of polynomials of all degrees, and $\mathcal{P}_m$ be the vector space of polynomials of degree $\leq m$.  For $p,q\in \mathcal{P}$ we define two pairings
\begin{align}
\label{scp1}
    \langle p,q\rangle_{pl}& = \int_{\mathbb C}p(z)\overline{q(z)}\frac{|z-w|^{2NQ_1}}{(1+|z|^2)^{N(1+Q_0+Q_1)+1}} \, dA(z),
    \\
    \langle p,q\rangle_{co}& =\oint_{\Gamma}p(z)q(z)\frac{(z-w)^{NQ_1}}{(1+wz)^{N+NQ_0}z^{n+NQ_1}} \, dz.   \label{scp2}
\end{align}
The pairings $\langle p,q\rangle_{pl},  \langle p,q\rangle_{co}$ are associated with planar orthogonality and non-Hermitian contour orthogonality respectively. Observe that \eqref{scp1} only depends on $N$, while \eqref{scp2} depends both on $N$ and $n$, although we do not include these in the notation. We define for $i,j=0,1,2,\dots$
\begin{equation}
\mu_{i,j} = \langle z^i, z^j \rangle_{pl}, \qquad \nu_i = \langle z^i, 1 \rangle_{co}.
\end{equation}

By Proposition \ref{contorth} the monic orthogonal polynomial $P_{n,N}$ of degree $n$ satisfies two determinantal formulas,
\begin{equation}\label{det1}
P_{n,N}(z) = \frac{1}{\det[\mu_{i,j}]_{i,j=0}^{n-1}} \begin{vmatrix}
\mu_{0,0} & \cdots & \mu_{0,n-1} & 1 \\
\mu_{1,0} & \cdots & \mu_{1,n-1} & z \\
\vdots & \ddots & \vdots & \vdots \\
\mu_{n,0} & \cdots & \mu_{n,n-1} & z^n
\end{vmatrix}
\end{equation}
\begin{equation}\label{det2}
\hspace*{1.25cm} = \frac{1}{\det[\nu_{i+j}]_{i,j=0}^{n-1}} \begin{vmatrix}
\nu_0 & \cdots & \nu_{n-1} & 1 \\
\nu_1 & \cdots & \nu_n & z \\
\vdots & \ddots & \vdots & \vdots \\
\nu_n & \cdots & \nu_{2n-1} & z^n
\end{vmatrix}.
\end{equation}
One obtains from \eqref{det1} and \eqref{det2} by standard theory of orthogonal polynomials
\begin{equation}\label{detrep1}
h_{n,N} = \langle P_{n,N}, z^n \rangle_{pl} = \frac{\det[\mu_{i,j}]_{i,j=0}^n}{\det[\mu_{i,j}]_{i,j=0}^{n-1}}, 
\end{equation}
\begin{equation}\label{detrep2}
\widetilde{h}_{n,N} = \langle P_{n,N}, z^n \rangle_{co} = \frac{\det[\nu_{i+j}]_{i,j=0}^n}{\det[\nu_{i+j}]_{i,j=0}^{n-1}}.
\end{equation}

Lemma \ref{contorth0} tells us that
\begin{equation}\label{screl}
\langle z^i, (z-w)^j \rangle_{pl} = \frac{G_{j,N}}{2i} \langle z^i, (1 + wz)^j z^{n-j-1} \rangle_{co}, \quad i,j = 0, \ldots, n.
\end{equation}
The polynomials $(z-w)^{j}, j=0,1,\dots n,$ is a basis of $\mathcal{P}_n$. If we change our basis to  $z^{j}, j=0,1,\dots n$ the change of basis matrix is a triangular matrix with ones on the diagonal. As a consequence we have the determinant to be $1$ for this matrix.  This along with \eqref{screl} gives us  
\begin{align}
\begin{split}
\label{frep1} 
\det[\mu_{i,j}]_{i,j=0}^n &=\det[\langle z^i,(z-w)^j\rangle_{pl}  ]_{i,j=0}^n\\& =\prod_{j=0}^n \frac{G_{j,N}}{2i}  \det[\langle z^i, (1 + wz)^j z^{n-j-1} \rangle_{co}]_{i,j=0}^n,
\end{split}
\\
\begin{split}
\label{frep2} 
\det[\mu_{i,j}]_{i,j=0}^{n-1}  &=\det[\langle z^i,(z-w)^j\rangle_{pl}  ]_{i,j=0}^{n-1}\\&= \prod_{j=0}^{n-1}  \frac{G_{j,N}}{2i} \det[\langle z^i, (1 + wz)^j z^{n-j-1} \rangle_{co}]_{i,j=0}^{n-1}. 
\end{split}
\end{align} 
Therefore, by \eqref{detrep1}, \eqref{frep1} and \eqref{frep2}
we obtain 
\begin{equation}\label{detrep3}
h_{n,N} = \frac{G_{n,N}}{2i}  \frac{\det[\langle z^i, (1 + wz)^j z^{n-j-1} \rangle_{co}]_{i,j=0}^n}{\det[\langle z^i, (1 + wz)^j z^{n-j-1} \rangle_{co}]_{i,j=0}^{n-1}}.
\end{equation}

As we have seen in the proof of Theorem \ref{contorth} the family of polynomials $(1+wz)^{j}z^{n-j-1}, j=0,1,\dots, n-1$, is a basis of $\mathcal{P}_{n-1}$. As before if we change to the  basis  $z^{n-1}, z^{n-2},\dots ,z,1$, the change of basis matrix is lower triangular with ones on the diagonal. As a result, we obtain 
\begin{equation}
\begin{aligned}\label{sim1}
\det[\langle z^i, (1 + wz)^j z^{n-j-1} \rangle_{co}]_{i,j=0}^{n-1}
&= \det[\langle z^i,  z^{n-j-1}  \rangle_{co}]_{i,j=0}^{n-1}\\&=(-1)^{n(n-1)/2} \det[\langle z^i, z^{j} \rangle_{co}]_{i,j=0}^{n-1}\\&= (-1)^{n(n-1)/2} \det[\nu_{i+j}]_{i,j=0}^{n-1}.
\end{aligned}
\end{equation}
For the second identity we reordered the columns of the matrix in the determinant.

 The functions $(1+wz)^{j}z^{n-j-1}, j=0,1,\dots, n$, (for $j=n$ it is not a polynomial), correspond in the same way as above to the basis $z^{n-1},\dots,z^{-1}$ and similar to \eqref{sim1}, we obtain 
\begin{equation}
\begin{aligned}\label{sim2}
\det[\langle z^i, (1 + wz)^j z^{n-j-1} \rangle_{co}]_{i,j=0}^{n}
&= \det[\langle z^i, z^{n-j-1} \rangle_{co}]_{i,j=0}^{n}\\&=(-1)^{n(n+1)/2} \det[\langle z^i, z^{j-1} \rangle_{co}]_{i,j=0}^{n}\\&= (-1)^{n(n+1)/2} \det[\nu_{i+j-1}]_{i,j=0}^{n}.
\end{aligned}
\end{equation}

We now relate $ \det[\nu_{i+j-1}]_{i,j=0}^{n}$ to $P_{n+1,N}(0)$. Recall \eqref{scp1} does not depend on $n$ and it is still valid if $n$ is replaced by $n+1$. It is not true for the formula \eqref{det2} since the pairing
\eqref{scp2} depends on $n$, and so do the moments $\nu_i$. In fact, by \eqref{scp2} the moments simply
go down by one unit if we go from $n$ to $n+1$, and the analogue of the formula \eqref{det2} is
\begin{equation}
P_{n+1,N}(z) = \frac{1}{\det[\nu_{i+j-1}]_{i,j=0}^n} \begin{vmatrix}
\nu_{-1} & \cdots & \nu_{n-1} & 1 \\
\nu_0 & \cdots & \nu_n & z \\
\vdots & \ddots & \vdots & \vdots \\
\nu_n & \cdots & \nu_{2n} & z^{n+1}
\end{vmatrix}.
\end{equation}
Evaluating at $z=0$,
\begin{equation}\label{p0}
P_{n+1,N}(0) = \frac{1}{\det[\nu_{i+j-1}]_{i,j=0}^n} \begin{vmatrix}
\nu_{-1} & \cdots & \nu_{n-1} & 1 \\
\nu_0 & \cdots & \nu_n & 0 \\
\vdots & \ddots & \vdots & \vdots \\
\nu_n & \cdots & \nu_{2n} & 0
\end{vmatrix}
= (-1)^{n+1} \frac{\det[\nu_{i+j}]_{i,j=0}^n}{\det[\nu_{i+j-1}]_{i,j=0}^n}.
\end{equation}
Multiplying \eqref{p0} and \eqref{detrep3} and using \eqref{detrep2} we obtain \eqref{h1}. 
\end{proof}

We are now ready to prove Theorem~\ref{norma}.

\begin{proof}[Proof of Theorem~\ref{norma}]

At first comparing \eqref{h0} with \eqref{RHnorm} we obtain 
\begin{equation}\label{h2}
    \widehat{h}_{n,N}:=\widetilde{h}_{n,N} w^{N+NQ_0}.
\end{equation}
Then the identity \eqref{h1} in Lemma \ref{rela1} can be written as 
\begin{equation}\label{hasy2}
    h_{n,N}=-\frac{1}{2i}G_{n,N}P_{n+1,N}(0)^{-1}w^{-N-NQ_0}\widehat{h}_{n,N}.
\end{equation}
It remains to find the asymptotic behaviour of $G_{n,N},P_{n+1,N}(0)$ and $\widehat{h}_{n,N}$ as $N\rightarrow\infty$ with $n=N+r_0$ fixed. We start with the latter quantity.

Due to \eqref{rhnorm} and \eqref{Yasy} we have 
\begin{equation}\label{hathasy}
\begin{aligned}
    \widehat{h}_{n,N}&=-2\pi i\lim_{z\rightarrow\infty}z^{n+1}\left[e^{N\frac{\ell_0}{2}\sigma_3}S(z)M(z)e^{N(g(z)-\frac{\ell_0}{2})\sigma_3}\right]_{1,2}\\&
   = -2\pi i e^{N\ell_0}\lim_{z\rightarrow\infty}z^{n+1}e^{-Ng(z)}\left[S(z)M(z)\right]_{1,2},
\end{aligned}
\end{equation}
where $S$ satisfies \eqref{Sasy}. From \eqref{gsymp} it follows that $ze^{-g(z)}\rightarrow1$ as $z\rightarrow\infty$. Using \eqref{newrhp} and the fact $n=N+r_0$ we conclude from \eqref{hathasy} that 
\begin{align}
    \widehat{h}_{n,N}&=-2\pi i e^{N\ell_0}\lim_{z\rightarrow\infty}z^{r_0+1}\left[S(z)D_{\infty}^{-r_0\sigma_3}N(z)D(z)^{r_0\sigma_3}\right]_{1,2} \notag \\&
    =-2\pi i e^{N\ell_0}\lim_{z\rightarrow\infty}z^{r_0+1}D(z)^{-r_0}\left[S(z)D_{\infty}^{-r_0\sigma_3}N(z)\right]_{1,2} \notag \\&
     =-2\pi i e^{N\ell_0}D^{-r_0}_{\infty}\lim_{z\rightarrow\infty}z \left[S(z)D_{\infty}^{-r_0\sigma_3}N(z)\right]_{1,2}. \label{1as}
\end{align}
To obtain the last equality we used Lemma \ref{defD} (3) which tells us $D(z)/z\rightarrow D_{\infty}$ as $z\rightarrow\infty$.

From \eqref{gpar}, \eqref{gpar3} and a direct calculation based on \eqref{invasy}
we find 
\begin{equation}
    N_{1,2}(z)=i\sqrt{\rho}\frac{\sqrt{F_2'(z)}}{F_2(z)}=-\rho \frac{\sqrt{a(b-a)}}{z}+\mathcal{O}(z^{-2})\quad \mathrm{as}\hspace{0.1cm} z\rightarrow\infty,
\end{equation}
while $N_{2,2}(z)=1+\mathcal{O}(z^{-1})$ because of the asymptotic property \eqref{Nasy1} in the RH problem \ref{RHN} for $N$.
The asymptotic formula \eqref{Sasy} for $S$ tells us $S_{1,1}(z)\rightarrow 1$ and $z S_{1,2}(z)=\mathcal{O}(N^{-1})$ as $z\rightarrow\infty$. Thus we find 
\begin{align*}
\lim_{z \rightarrow \infty} z \left[ S(z) D_{\infty}^{-r_0 \sigma_3} N(z) \right]_{1,2} 
&= \lim_{z \rightarrow \infty} z \left[ 
    S_{1,1}(z) D_{\infty}^{-r_0} N_{1,2}(z) 
    + S_{1,2}(z) D_{\infty}^{r_0} N_{2,2}(z) 
\right] \\
&= -D_{\infty}^{-r_0} \rho \sqrt{a(b-a)} + \mathcal{O}(N^{-1}).
\end{align*}
Combining this with \eqref{1as} and Lemma \ref{defD} (3) (which tells us $D_{\infty}^{-2}=\rho a$) we obtain 
\begin{equation}\label{2as}
    \widehat{h}_{n,N}=2\pi i e^{N\ell_0}(\rho a)^{r_0}\rho\sqrt{a(b-a)}\left(1+\mathcal{O}(N^{-1})\right).
\end{equation}

Now, it remains to compute the asymptotic expansions of $P_{n+1,N}(0)$ and $G_{n,N}$.
We use the asymptotic formula \eqref{strongasy} with $z=0$ to obtain 
$$P_{n+1,N}(0)=i\left(\frac{\rho}{F_{1}(0)}\right)^{r_0+1}\frac{\sqrt{\rho F_1'(0)}}{F_1(0)}\left(1+\mathcal{O}(N^{-1})\right).$$
We recall $F_1$ is the inverse function of $f$, and so $F_1(0)=1/b$ as $f(1/b)=0$. From \eqref{confm} one obtains $f'(1/b)=-\frac{b-a}{b^3\rho}$, therefore, by the rule for the derivative of the inverse function one has $F_1'(0)=-\frac{b^3\rho}{b-a}$. It then follows that 
\begin{equation}\label{p0exp}
    P_{n+1,N}(0)=-\rho \left(\rho b\right)^{r_0}\sqrt{b(b-a)}e^{Ng(0)}\left(1+\mathcal{O}(N^{-1})\right).
\end{equation}

By applying Stirling's formula in \eqref{defbeta}, for $n=N+r_0$, a straightforward calculation yields
\begin{equation}\label{gexp1}
    G_{n,N}=\sqrt{\frac{2\pi}{N}}\frac{Q_0^{NQ_0-r_0-\frac{1}{2}}(1+Q_1)^{N(1+Q_1)+r_0+\frac{1}{2}}}{(1+Q_0+Q_1)^{N(1+Q_0+Q_1)-\frac{1}{2}}}\left(1+\mathcal{O}(N^{-1})\right).
\end{equation}
Then inserting \eqref{gexp1}, \eqref{p0exp} and \eqref{2as} in \eqref{hasy2} we find 
\begin{align*}
h_{n,N} &= \pi \sqrt{\frac{2\pi}{N}} 
\frac{
    Q_0^{NQ_0 - r_0 - \frac{1}{2}} 
    (1 + Q_1)^{N(1 + Q_1) + r_0 + \frac{1}{2}}
}{
    (1 + Q_0 + Q_1)^{N(1 + Q_0 + Q_1) - \frac{1}{2}}
}
\\
&\quad \times \left( \frac{a}{b} \right)^{r_0 + \frac{1}{2}} 
e^{N(\ell_0 - g(0))} 
w^{-N - NQ_0} 
\left(1 + \mathcal{O}(N^{-1})\right).
\end{align*}

Finally, with \eqref{relation} and \eqref{a b ratio Q0Q1} this greatly simplifies to \eqref{norma2}. This completes the proof of Theorem \ref{norma}. 
\end{proof}

\subsection*{Acknowledgements}
Sung-Soo Byun was supported by the New Faculty Startup Fund at Seoul National University and by the National Research Foundation of Korea grant (RS-2023-00301976, RS-2025-00516909).
Funding support to Peter Forrester for this research was through the Australian Research Council Discovery Project grant DP250102552. Arno Kuijlaars is supported by long term structural funding-Methusalem grant of
the Flemish Government, and by FWO Flanders project G.0910.20.
Sampad Lahiry acknowledges financial support from the International Research Training Group (IRTG) between KU Leuven and University of Melbourne and Melbourne Research Scholarship of University of Melbourne.

\appendix

\end{document}